\documentclass[12pt]{amsart}
\usepackage[active]{srcltx}
\usepackage{amsmath}
\usepackage{amsthm}
\usepackage{amssymb}
\usepackage{fullpage}
\usepackage[all,cmtip]{xy}
\usepackage{amsfonts}
\usepackage{verbatim}
\usepackage{amscd,latexsym}
\usepackage{tikz}
\usepackage{tikz-cd}
\usepackage{caption}
\usepackage{array}
\usepackage{multirow}
\usepackage{mathrsfs}
\usepackage{hyperref}
\usepackage{tabularx}
\usepackage{pgfplots}
\usepackage{float}
\usepackage{bm}
\usepackage{graphicx}
\usepackage{makecell}
\usepackage{longtable}
\usepackage{mathabx}
\usepackage{tabstackengine}
\usepackage{mathtools}
\usepackage{array}
\newsavebox{\mybox}
\usepackage{comment}
\usepackage[utf8]{inputenc}
\usepackage[T1]{fontenc}
\hypersetup{colorlinks=true, linkcolor=blue, filecolor=magenta, urlcolor=cyan, citecolor=purple}
\newtheorem{theorem}{Theorem}
\numberwithin{theorem}{section}
\newtheorem{corollary}[theorem]{Corollary}
\newtheorem{proposition}[theorem]{Proposition}
\newtheorem{lemma}[theorem]{Lemma}
\theoremstyle{definition}
\numberwithin{theorem}{section}
\newtheorem{remark}[theorem]{Remark}

\newcommand{\C}{\mathbb{C}}

\newcommand{\F}{\mathbb{F}}

\newcommand{\bt}{\begin{theorem}}
	\newcommand{\et}{\end{theorem}}
\newcommand{\bd}{\begin{definition}}
	\newcommand{\ed}{\end{definition}}

\newcommand{\one}{1\!\!1}

\newcommand{\ind}{{\rm{ind}}}
\newcommand{\Irr}{{\rm{Irr}}}
\newcommand{\Alg}{{\rm{Alg}}}

\newcommand{\Ind}{{\rm{Ind}}}

\newcommand{\diag}{{\rm diag}}

\newcommand{\tr}{{\rm tr}}

\usepackage[margin=2.2cm,top=2.5cm,headheight=16pt,headsep=0.3in,heightrounded]{geometry}

\usepackage{fancyhdr}
\numberwithin{equation}{section}

\title {Structure of Twisted Jacquet Modules of principal series representations of $Sp_{4}(F)$}

\author{Sanjeev Kumar Pandey}
\address{Sanjeev Kumar Pandey, Department of Mathematics, Indian Institute of Science Education and Research Tirupati, Srinivasapuram, Jangalapalli Village, Panguru (G.P),
	Yerpedu Mandal, Tirupati Dist., Andhra Pradesh,
	India – 517619.}
\email{sanjeevkumarpandey@students.iisertirupati.ac.in}

\author{C. G. Venketasubramanian}
\address{C. G. Venketasubramanian, Department of Mathematics, Indian Institute of Science Education and Research Tirupati, Srinivasapuram, Jangalapalli Village, Panguru (G.P),
	Yerpedu Mandal, Tirupati Dist., Andhra Pradesh,
	India – 517619.}
\email{venketcg@iisertirupati.ac.in}

\begin{document}

	\maketitle

	\begin{abstract}  
		{\footnotesize   Let $F$ be a non-archimedean local field. For the symplectic group $Sp_{4}(F),$ let $P$ and $Q$ denote respectively its Siegel and Klingen parabolic subgroups with respective Levi decompositions $P=MN$ and $Q=LU.$ For a non-trivial character $\psi$ of the unipotent radical $N$ of $P,$ let $M_{\psi}$ denote the stabilizer of the character $\psi$ in $M$ under the conjugation action of $M$ on characters of $N.$ For an irreducible representation of the Levi subgroups $M$ or $L,$ let $\pi$ denote the respective representation of $Sp_{4}(F)$ parabolically induced either from $P$ or from $Q.$ Let $\psi$ be a character of the group $N$ given by a rank one quadratic form. In this article, we determine the structure of the twisted Jacquet module $r_{N,\psi}(\pi)$ as an $M_{\psi}$-module. We also deduce the analogous results in the case where $F$ is a finite field of order $q.$
		}
	
\end{abstract}

\tableofcontents

\section{Introduction}\label{introduction}

Let $F$ be a non-archimedean local field or a finite field $\F_q.$ Let $G$ be the $F$-points of a connected reductive group defined over $F.$ Let $P$ be a parabolic subgroup of $G$ with Levi 
decomposition $P = M N.$ Let $(\pi,V)$ be any smooth representation 
of $G$ and $\psi$ be any non-trivial character of $N.$ Let $V_{N,\psi}$ 
denote the largest quotient of $V$ on which $N$ operates by $\psi.$ The quotient  $V_{N,\psi}$ is a representation of the subgroup $M_\psi $ of 
$M$  which stabilizes the character $\psi.$ This representation of $M_{\psi}$ on $V_{N,\psi}$ is denoted by
$r_{N,\psi}(\pi)$ and is called the twisted Jacquet module of $\pi.$ It is a natural and interesting question
to understand for which irreducible representations $\pi$ of $G$ the twisted Jacquet 
module $r_{N,\psi}(\pi)$ is non-zero and to determine the structure of $r_{N,\psi}(\pi)$ as an  $M_\psi$-module 
 whenever it is non-zero. In this article, we shall be addressing the question for parabolically induced representations of the symplectic group $Sp_{4}(F)$ with respect to a character $\psi$ of the unipotent radical of its Siegel parabolic subgroup where $\psi$ is given by a rank one quadratic form.

Primarily, the question of determining the structure of twisted Jacquet modules for parabolically induced representations of $p$-adic groups can be considered as a problem in the general setting of the Geometric lemma (see \cite{BZ1} and \cite{BZ2}) of Bernstein-Zelevinsky. On the other hand, computing the structure of twisted Jacquet modules for specific groups and representations have found applications in the study of branching laws of representations of $p$-adic groups and especially in the context (see for example \cite{DPSarah}, \cite{GS}, \cite{DPRamin}, \cite{DPJAlgmult}, \cite{DPDivAlgJNT}, \cite{DPRaghuramDuke}  and \cite{Brooks-Schmidt}) of distinction problems of representations of $p$-adic groups. 

We briefly recall certain contexts in which the structure of twisted Jacquet modules have been determined. In \cite{DPDegIMRN}, Prasad explicitly calculates the structure of the twisted Jacquet module of a cuspidal representation of $GL_{2n}(\mathbb{F}_q),$ with respect to a non-degenerate character of the unipotent radical of the maximal parabolic $P_{n,n}(\mathbb{F}_q).$ In \cite{OfirZahi}, the work of Prasad was extended to $GL_{k \cdot n}(\F_q).$ In recent work, K. Balasubramanian et. al (see \cite{KH1}, \cite{KH2} and \cite{KH3}) have further extended the work of Prasad by determining the twisted Jacquet module of a cuspidal representation of $GL_{2n}(\mathbb{F}_q)$ for a degenerate character of the unipotent radical given by a rank one matrix. For principal series representations, the structure of twisted Jacquet modules have been considered in  \cite{GS}, \cite {DPRamin}, \cite{DPDegIMRN}, \cite{DPDegTIFR} and \cite{Brooks-Schmidt}.

We would like to emphasize that our present work is motivated by the work of Prasad-Takloo-Bighash \cite{DPRamin} and Brooks-Schmidt \cite{Brooks-Schmidt} undertaken for the group $GSp_{4}(F)$ over a $p$-adic field. We wish to remark here that in \cite[\S 2]{DPRamin} and \cite[\S 5]{Brooks-Schmidt}, the character $\psi$ of the unipotent radical of the Siegel parabolic subgroup of $GSp_{4}(F)$ is given by a rank two quadratic form and it has applications in the context of the theory of Bessel models of representations of $GSp_{4}(F).$ For more detailed comments on the results in \cite[\S 2]{DPRamin}, we refer the reader to \cite[\S 5]{Brooks-Schmidt}.  In this article, we work over $Sp_{4}(F)$ and the character of the unipotent radical of the Siegel parabolic of $Sp_4(F)$ is given by a  rank one quadratic form in contrast to a rank two quadratic form considered in \cite{DPRamin} and \cite{Brooks-Schmidt}.

We fix minimal notation so as to state our main results. Let $F$ be a non-archimedean local field with characteristic not equal to $2.$ Let $|~|_{_F}$ denote the standard normalized absolute
value on $F.$ Let  $Sp_{4}(F)$ denote the symplectic group attached to the non-degenerate alternating form given by $\left[\begin{array}{cc}
0&I_{2}  \\
-I_{2}&0 
\end{array}
\right].$ 
Let $P$ denote the Siegel parabolic subgroup of $Sp_4(F)$ with Levi decomposition $P=MN$ where the Levi subgroup $M$ is identified with $GL_{2}(F)$ and the unipotent radical $N$ is identified with $SM_{2}(F),$ the vector subspace consisting of all symmetric $2\times 2$ matrices of the vector space $M_2(F).$  Let $Q$ denote the Klingen parabolic subgroup of $Sp_4(F).$ The group $Q$ has a Levi decomposition $Q=LU,$ where the Levi subgroup $L$ can be identified with $F^{\times}\times SL_{2}(F)$ and $U$ is the unipotent radical of $Q.$ 

Fix a non-trivial character $\psi_{_0}$ of the additive group $F.$ 
Any character of $N$ is of the form $\psi_{_A},$ for some $A \in SM_2(F)$ where $\psi_{_A}$ is given by $\left[\begin{array}{cc}
	I_2 & X \\
	0 & I_2
\end{array}\right] \mapsto \psi_{_0}(\tr(AX))$ for all $X \in SM_2(F).$ If $A$ is of rank one, then there exists $g \in GL_2(F)$ such that  $^tg A g =\left[\begin{array}{cc}
		\gamma & 0 \\
		0 & 0
	\end{array}\right]$ for some $\gamma \in F^{\times}.$ Consequently, if we set $C = \left[\begin{array}{cc}
	\gamma & 0 \\
	0 & 0
\end{array}\right]$ then  $\psi_{_A}$ is in the $M$-orbit of the character $\psi_{_C}$ of $N$ under conjugation action of $M$ on characters of $N,$ where  \begin{equation}\label{CharacterCorrspondngtoC}
\psi_{_C}\left(\left[\begin{array}{cc}
I_{2}&\left[\begin{array}{cc}
x & y \\
y & z
\end{array}\right] \\
0&I_{2} 
\end{array} \right]
\right)=\psi_{_0}\left(\tr\left(C\left[\begin{array}{cc}
x & y \\
y & z
\end{array}\right]\right) \right)=\psi_{_0}(\gamma x),
\end{equation} for every $x, y, z \in F.$ 

Given an irreducible representation $\rho$ of $M$ and $\varrho$ of $L$, let $\pi$ and $\Pi$ denote the normalized  parabolically induced representations $\ind_P^{Sp_4(F)}(\rho)$ and $\ind_Q^{Sp_4(F)}(\varrho)$ respectively.  We determine the structure of the twisted Jacquet module $r_{N,\psi_{_C}}(\pi)$ and $r_{N,\psi_{_C}}(\Pi).$  As in \cite{BZ2}, our twisted Jacquet modules are normalized. 

Let \begin{equation}\label{N_1,1andbardefn}
N_{1, 1} = \left\{
\left[\begin{array}{cc}
1 & x \\
0 & 1
\end{array}\right] : x\in F \right\} \text{ and }
\overline{N_{1, 1}} = \left\{
\left[\begin{array}{cc}
1 & 0 \\
x & 1
\end{array}\right] : x\in F \right\}
\end{equation} 
 denote the upper triangular unipotent subgroup of $GL_{2}(F)$ and its opposite respectively. 
 
 For any smooth representation $\rho$ of $GL_{2}(F),$ let $r_{\overline{N_{1, 1}}}(\rho)$ denote the Jacquet module of $\rho$  with respect to $\overline{N_{1,1}}.$
 For $\gamma \in F^{\times},$ let $\psi_{\gamma}: N_{1,1} \rightarrow \C^{\times}$ be the character of $N_{1,1}$ defined by
 \begin{equation}\label{Definitionpsi_x}
 	\psi_{\gamma}\left( \left[\begin{array}{cc}
 		1 & x \\
 		0 & 1
 	\end{array}\right]\right) = \psi_{_0}(\gamma x),
 \end{equation} 
 for every $x \in F.$
For any smooth representation $\tau$ of $SL_{2}(F),$ let $ r_{N_{1,1}}(\tau)$ denote the Jacquet module of $\tau.$ Moreover, let $r_{N_{1,1}, \psi_{\gamma}}(\tau)$ denote the twisted Jacquet module of $\tau$ with respect to $N_{1,1}$ and the non-trivial character $\psi_{\gamma}.$  
 We recall that, if $r_{N_{1,1},\psi_{\gamma}}(\tau)\neq 0$ then it acts as $\omega_{\tau},$ the central character of $\tau,$   which can be thought of as a representation of the group $\{\pm 1\}.$ Let $\pi$ be a smooth representation of an $\ell$-group $G.$ Let $\theta$ be a character of $G.$ We denote by $\theta \cdot \pi$ the representation of $G$ which is defined by $(\theta \cdot \pi)(g) = \theta(g) \cdot \pi(g)$ for $g \in G.$ 
We shall denote the contragredient (smooth dual) of $\pi$ by $\widetilde{\pi}.$  
 
 Let  $O_{2}(F,C)$ denote the orthogonal group preserving the symmetric bilinear form given by the matrix $C,$ i.e.,
 \begin{equation}\label{O2(F,C)-defn}
 	O_{2}(F,C) := \left\{\left[\begin{array}{cc} \pm 1 & 0 \\
 		y & d\end{array}\right] : d \in F^{\times}, y \in F \right\}.
 \end{equation}Let \begin{equation}\label{DefofT2FC}
T_{2}(F,C) := \left\{\left[\begin{array}{cc} \pm 1 & 0 \\0 & d\end{array}\right] : d \in F^{\times} \right\}
\end{equation} 
be the subgroup of $O_{2}(F,C)$ which is isomorphic to the group $ \{\pm 1\} \times F^{\times}.$ Note that $O_2(F,C)=T_2(F,C) \ltimes \overline{N_{1,1}}.$ Any representation of $T_2(F,C)$ can therefore be inflated to get a representation of $O_2(F,C)$ by composing with the quotient map from $O_2(F,C)$ to $T_2(F,C).$  The group $O_2(F,C)$ is not unimodular and its modular character is given by $\delta_{_{O_{2}(F,C)}}\left(\left[\begin{array}{cc}
	\pm 1 & 0 \\
	y & d
	\end{array}\right]\right) = |d|_{_F}$ for  $d \in F^{\times}$ and $y\in F.$

We now state our first result which gives the structure of the twisted Jacquet module of $\ind_P^{Sp_4(F)}(\rho)$ with respect to the character of $N$ defined by \eqref{CharacterCorrspondngtoC}. 

\begin{theorem}\label{SiegelcaseforDegA}
	Fix a non-trivial character $\psi_{_0}$ of the additive group $F.$ Let $P=MN$ 
	denote the Siegel parabolic subgroup of the symplectic group $Sp_{4}(F)$ and 
	let $\rho\in \Irr(GL_{2}(F))$ be regarded as a representation of the Levi 
	subgroup $M$ of $P.$ Denote $\pi=\ind_P^{Sp_{4}(F)}(\rho).$ Let $\psi_{_{C}}$ be the character of $N$ defined by \eqref{CharacterCorrspondngtoC}. The following statements hold:
	
	\begin{enumerate}
		\item If $\rho$ is not cuspidal,  $r_{N,\psi_{_{C}}}(\pi)$
		sits in the following short exact sequence of $O_{2}(F,C)$-modules:
		\begin{equation*}
		0\to  \widetilde{\rho}_{|_{O_{2}(F,C)}}\to r_{N,\psi_{_{C}}}(\pi)\to \rho_{_0} \to 0,
		\end{equation*}
		where $\rho_{_0}$ is the representation of $O_{2}(F,C)$ obtained from $(\delta_{_{O_{2}(F,C)}}^{\frac{1}{2}}  \cdot r_{\overline{N_{1, 1}}}(\rho))_{|_{T_{2}(F,C)}}$ via composing with the quotient map.

		\item If $\rho$ is cuspidal, $r_{N,\psi_{_{C}}}(\pi) = \widetilde{ \rho}_{|_{O_{2}(F,C)}}$ as $O_{2}(F,C)$-modules. 
	\end{enumerate}
	
\end{theorem}

The next theorem deals with the Klingen parabolic subgroup and is the second main result of the article.

\begin{theorem}\label{KlingencaseforDegA}
	Fix a non-trivial character $\psi_{_0}$ of the additive group $F.$ Let $Q = L U$ denote the Klingen parabolic subgroup of the symplectic group $Sp_{4}(F).$ Suppose $\varrho = \eta \otimes \tau \in \Irr(L)$ where $\eta$ is a character of $F^{\times}$ and  $\tau \in \Irr(SL_{2}(F)).$ Denote $\Pi = \ind_{Q}^{Sp_{4}(F)}\left(\varrho\right).$ Let $\psi_{_C}$ be the character of $N$ defined by \eqref{CharacterCorrspondngtoC}. 
	Let $\varrho_{_0}$ and $\varrho_{_1}$ denote the representations of $O_{2}(F,C)$ obtained respectively from the representations $\delta_{_{O_{2}(F,C)}}^{\frac{1}{2}} \cdot (r_{N_{1,1}, \psi_{\gamma}}(\tau) \otimes \eta)$ and $\delta_{_{O_{2}(F,C)}}^{\frac{1}{2}} \cdot  (\eta_{_{|_{\{\pm{1}\}}}}\otimes  r_{N_{1,1}}(\tau))$ of $T_2(F,C)$ via composing with the quotient map. Let $\varrho_{_2}$ denote the representation
	$ r_{N_{1,1}, \psi_{\gamma}}(\tau) \otimes \eta^{-1}$ of $T_2(F,C).$
	Then, $r_{N,\psi_{_C}}(\Pi)$ is glued from the representations 
	$\varrho_{_0},$ $\varrho_{_1}$ and 	$\ind_{T_{2}(F,C)}^{O_{2}(F,C)}(\varrho_{_2})$ of $O_{2}(F,C).$

	\end{theorem}

The following corollary of Theorem \ref{KlingencaseforDegA}   describes the finer structure of the twisted Jacquet module with respect to the inducing data $\tau.$

\begin{corollary}\label{Cortoklingen}
	Fix a non-trivial character $\psi_{_0}$ of the additive group $F.$ Let $Q = L U$ denote the Klingen parabolic subgroup of the symplectic group $Sp_{4}(F).$ Suppose $\varrho = \eta \otimes \tau \in \Irr(L)$ where $\eta$ is a character of $F^{\times}$ and  $\tau \in \Irr(SL_{2}(F)).$ Denote $\Pi = \ind_{Q}^{Sp_{4}(F)}\left(\varrho\right).$ Let $\psi_{_C}$ be the character of $N$ defined by \eqref{CharacterCorrspondngtoC}. Let $\omega_{\tau}$ denote the central character of $\tau.$ Define $\varrho_{_{0}}$ and $\varrho_{_{2}}$ as follows: $\varrho_{_{0}}$ is the representation of $O_{2}(F,C)$ obtained from the representation $ \delta_{_{O_{2}(F,C)}}^{\frac{1}{2}} \cdot (\omega_{\tau} \otimes \eta)$ of $T_{2}(F,C)$ via composing with the quotient map and $\varrho_{_{2}}$  is the representation $\omega_{\tau} \otimes \eta^{-1} $  of $T_2(F,C).$ Let $\varrho_{_1}$ be the representation of $O_2(F,C)$ obtained from the representation $\delta_{_{O_{2}(F,C)}}^{\frac{1}{2}} \cdot (\eta_{|_{\{\pm 1\}}}\otimes r_{N_{1,1}}(\tau) )$ of $T_{2}(F,C)$ via composing with the quotient map. The following statements hold:
	
	\begin{enumerate}

		\item If $r_{N_{1,1},\psi_{\gamma}}(\tau)=0$ and $\tau$ is cuspidal,  $r_{N,\psi_{_C}}(\Pi)=0.$
		
		\item If $r_{N_{1,1},\psi_{\gamma}}(\tau)=0$ and $\tau$ is not cuspidal, $r_{N,\psi_{_C}}(\Pi)= \varrho_{_1}.$ In particular, if $\tau$ is one dimensional, $r_{N,\psi_{_C}}(\Pi)=\varrho_{_1}.$
		
		\item Assume that $\tau$ is cuspidal and $r_{N_{1,1},\psi_{\gamma}}(\tau)\neq 0.$ Then, $r_{N,\psi_{_C}}(\Pi)$ sits in the following short exact sequence of $O_{2}(F,C)$-modules:
		\begin{equation*}
		0\to  \ind_{T_{2}(F,C)}^{O_{2}(F,C)}(\varrho_{_{2}} ) \to r_{N,\psi_{_C}}(\Pi)\to  \varrho_{_{0}} \to 0.
		\end{equation*}
		
		\item Assume that $\tau$ is not cuspidal and $r_{N_{1,1},\psi_{\gamma}}(\tau)\neq 0.$ Then, $r_{N,\psi_{_C}}(\Pi)$ is given by a filtration
		\begin{equation*}
		\{0\} \subset W_2\subset W_1\subset W_0=r_{N,\psi_{_C}}(\Pi)
		\end{equation*}
		where $W_0/W_1=  \varrho_{_{0}},$ $W_1/W_2=  \varrho_{_1}$ and $W_2=\ind_{T_{2}(F,C)}^{O_{2}(F,C)}(\varrho_{_{2}}).$
		
	\end{enumerate}

\end{corollary}

In this paragraph, we make few pertinent remarks on the structure of our computation. Let $S_{\psi}$ denote the stabilizer of the character $\psi$ in $P.$  The general structure of the twisted Jacquet module of a principal series representation can be obtained via  Geometric Lemma. However, to apply the Geometric Lemma, we have adopted a two step process. We first restrict a principal series representation to the Siegel parabolic $P$ and using the inducing subgroups therein, the orbits for the action of $S_{\psi}$ on $P \backslash Sp_{4}(F)$ and $Q \backslash Sp_{4}(F)$ are determined. These orbits are relevant for applying the Geometric Lemma and are achieved via Propositions \ref{S-psi-orbit-Sp4/P} and \ref{S-psi-orbit-Sp4/Q} using Lemma \ref{hijgj}, Propositions \ref{Pdoublecoset} and \ref{PdoublecosetGeneral}. With Propositions \ref{S-psi-orbit-Sp4/P} and \ref{S-psi-orbit-Sp4/Q} in place, we verify one technical condition, namely, decomposability of certain subgroups. This is achieved in Section \ref{Decomposability} via Lemmas \ref{Decomposability-verification-Siegel} and \ref{Decomposability-verification-Klingen}. The second step consists of applying the Geometric lemma and the finer structure of the twisted Jacquet modules are obtained using some standard arguments. A two step process is also adopted in \cite[\S 2, Theorem 2]{DPDegTIFR}, \cite[\S 5]{Brooks-Schmidt} and \cite[\S 7, Propositon 7.1]{GS}, although the second step in these works are distinct from the second step of our approach. To add further, in the above mentioned articles, the case of  a non-degenerate character is undertaken in contrast to the degenerate character that we have taken in our case.   

Before we end the introduction, we shall outline the contents of this article. In Section \ref{notation}, we set up notations that will be used throughout in the article. We prove some preparatory results in Section \ref{prelims} that will be used in the proof of our main theorems. We also review the Geometric lemma of Bernstein-Zelevinsky in this section. In Section \ref{Orbits-Spsi-action}, we obtain the double cosets relevant for the computation and verify the decomposability conditions. Following this, in  Section \ref{DegenerateCharacter}, we prove the main theorems in the $p$-adic field case. Finally, in Section \ref{finite-field-case}, we prove the analogue of our main results in the case where the groups are defined over a finite field.

\section{Notation}\label{notation}

Let $F$ be a non-archimedean local field or a finite field. We will assume throughout the article that the characteristic of $F$ is not equal to $2$.
Denote by $M_{m,n}(F)$ the $F$-vector space of all $m\times n$ matrices over $F.$ If $m=n,$ we shall denote $M_{n,n}(F)$ simply by $M_n(F).$ By $SM_n(F),$ we mean the subspace of $M_n(F)$ consisting of all symmetric $n\times n$ matrices. As usual, $GL_{n}(F)$ shall denote the group of all invertible $n\times n$ matrices over $F$ and $SL_{n}(F)$ denotes its subgroup consisting of elements with determinant one. Given any $x\in M_{m,n}(F),$ we shall denote its transpose by $~^{t}x.$ Given any vector $v$ in $F^{n},$ we think of $v$ as a column vector when we have to represent it as a matrix. For $\alpha_1,\dots, \alpha_n\in F,$ we shall denote the diagonal $n\times n$ matrix with diagonal entries $\alpha_1,\dots, \alpha_n$ by $\diag(\alpha_1,\dots, \alpha_n).$  More generally,  if $g_{j}\in M_{n_j}(F)$ for $1\leq j\leq k,$  we shall denote by $\diag(g_1,\dots, g_k)$ the block diagonal matrix with entries $g_1,\dots, g_k.$ 

Fix a positive integer $j$ such that $1<j<n.$ We shall denote by $P_{j,n-j}$ the standard maximal parabolic subgroup of $GL_n(F)$ corresponding to the partition $(j,n-j).$ Let $\overline{P_{{j, n-j}}}$ denote the parabolic subgroup of $GL_n(F)$ opposite to $P_{j,n-j}$ in $GL_{n}(F).$

\subsection{Generalities on $\ell$-groups}\label{Generalities}
We shall follow the terminology of \cite{BZ1} and \cite{BZ2}. Let $G$ be an $\ell$-group and $H$ be a closed subgroup of $G.$ Suppose $g\in G$ normalizes $H.$ Then, ${\rm mod}_H(g)$ is the module of the automorphism $h\mapsto ghg^{-1}$ of $H$ i.e., ${\rm mod}_H(g)$ is given by (see \cite[$\S 1$]{BZ2})
\begin{equation}\label{modular-character-defn}
	\displaystyle\int_H f(g^{-1}hg) d\mu(h)={\rm mod}_H(g) \displaystyle\int_H f(h) d\mu(h) 
\end{equation}
for all locally constant and compactly supported functions $f$ on $H$ where $d\mu$ denotes a left Haar measure on $H.$ Moreover, ${\rm mod}_H$ is a character of the normalizer of $H$ in $G.$ The character ${\rm mod}_G$ of $G$ is called the modular character of $G$ and will be denoted by $\delta_G.$

Given an $\ell$-group $G,$ we shall denote the 
category of smooth complex representations of $G$ by $\Alg(G).$  By $\Irr(G),$ 
we shall mean the set of irreducible admissible representations in $\Alg(G).$  For a closed subgroup $H$ of $G,$ the exact functors (\cite[\S 2.21 Remark]{BZ1}) of (normalized) smooth and compact induction (i.e., subrepresentation of smooth induction consisting of functions which are compactly supported modulo $H$) shall be denoted by $\Ind_H^G$ and $\ind_{H}^{G}$ respectively. The corresponding unnormalized inductions (\cite[\S 2.21]{BZ1}) shall be denoted by $I_H^{G}$ and $i_H^{G}$ respectively. We observe that $\ind_H^G(\rho)=i_H^G( \delta_{G}^{-\frac{1}{2}}\delta_H^{\frac{1}{2}} \cdot \rho)$ for any $\rho\in \Alg(H)$, the analogue of which holds for $\Ind_{H}^G$ also. 

\subsubsection{The functors $i_{_{\mathcal{U},\one}}$ and $r_{\mathcal{U},\theta}$}\label{BZinduction}

We recall two functors from  \cite[\S 1]{BZ2}. Let $\mathcal{U}$ be a closed subgroup of an $\ell$-group $G.$ Assume that $\mathcal{U}$ is a union of its open compact subgroups and  $\theta: \mathcal{U}\to \C^{\times}$ is a character.  Let $\mathcal{M}$ be a closed subgroup of $G$ such that $\mathcal{M}$ normalizes $ \mathcal{U}$ and $\theta.$  Assume also that  $\mathcal{M} \cap  \mathcal{U} = \{1\}$ and $\mathcal{P}= \mathcal{M}\mathcal{U}$ is closed in $G.$ We have a functor $i_{_{ \mathcal{U},\one}}:\Alg(\mathcal{M})\to \Alg(G)$ by \cite[\S 1.8]{BZ2}. Given any $(\rho,W)\in \Alg(\mathcal{M}),$ $i_{_{\mathcal{U},\one}}(\rho)$ is the representation of $G$ which acts by right translation on the vector space 

$$ \left\{f : G \to W ~\vline
\begin{array}{lll}
	(1) & f(mug) = {\rm mod}_{\mathcal{U}}^{\frac{1}{2}}(m)\rho(m)(f(g)),\ \forall m \in \mathcal{M}, \forall u\in \mathcal{U}, \
	\forall g \in G\\
	(2) & \mbox{$f(gk) = f(g)$ for all $g\in G,k \in K_f$ an open subgroup in $G$} \\
	(3) & f \mbox{ is compactly supported modulo } \mathcal{P}
\end{array}\right\}.$$ 
We note here from \cite[Proposition 29, Pg. 100]{Nachbin} that for a semi-direct product $\mathcal{P}=\mathcal{M}\mathcal{U}$ of locally compact groups, one has $\delta_{\mathcal{P}}(mu)=
{\rm mod}_{\mathcal{U}}(m)\delta_{\mathcal{M}}(m)\delta_{\mathcal{U}}(u)$ for each $m\in \mathcal{M}$ and $u\in \mathcal{U}.$ If $G$ is the $F$-points of a connected reductive group and $\mathcal{P}= \mathcal{M}\mathcal{U}$ is a parabolic subgroup of $G$ with Levi subgroup $\mathcal{M}$ and unipotent radical $\mathcal{U}$ then  ${\rm mod}_{\mathcal{U}}=\delta_{\mathcal{P}}.$ Consequently, $i_{_{\mathcal{U},\one}}(\rho)=\ind_{\mathcal{P}}^G(\rho)$ (normalized parabolic induction), for any $\rho\in \Alg(\mathcal{M})$ regarded as a representation of $\mathcal{P}$ whose action on $\mathcal{U}$ is trivial.

For $(\pi, V)\in \Alg(G),$  let $V(\mathcal{U},\theta)$ denote the span of $\{\pi(u)v-\theta(u)v: u\in  \mathcal{U}, v\in V\}.$ The group $\mathcal{M}$ acts on the quotient 
$V_{ \mathcal{U},\theta}:=V/V( \mathcal{U},\theta)$ as follows: for $m\in \mathcal{M}$ and $v\in V,$
$m\cdot (v+V( \mathcal{U},\theta))={\rm{mod}}_{ \mathcal{U}}^{-\frac{1}{2}}(m)\pi(m)v+V( \mathcal{U},\theta).$ This representation is denoted by $r_{ \mathcal{U},\theta}(\pi)$ and is called the {\it twisted Jacquet module} of $\pi.$  One thus obtains the (normalized) twisted Jacquet functor $r_{ \mathcal{U},\theta}:\Alg(G)\to \Alg(\mathcal{M}).$ If $\theta= \one,$ we shall denote $r_{ \mathcal{U}, \one}$ by $r_ \mathcal{U}$ (called the Jacquet functor). The twisted Jacquet functor is exact.

\begin{remark}\label{NormalizationFactor} 
The following two observations (see \cite[Pg. 445]{BZ2}) shall be used later.
\begin{enumerate}
\item If $\mathcal{U} = \{1\},$ then for any $\rho\in \Alg(\mathcal{M})$ we have $i_{_{\mathcal{U},\one}}(\rho) = i_{\mathcal{M}}^{G}(\rho),$ usual unnormalized induction.  Moreover, if $\mathcal{M} \mathcal{U} = G$ (i.e., $G=\mathcal{P}$)  then  $i_{_{ \mathcal{U},\one}}(\rho) =  {\rm{mod}}_{ \mathcal{U}}^{\frac{1}{2}}\cdot \rho $ regarded as a representation of $\mathcal{P}$ with the action of $\mathcal{U}$ being trivial, for any $\rho\in \Alg(\mathcal{M}).$

		\item If $\mathcal{U} = \{1\},$ then for any $\pi\in \Alg(G),$ $r_{\mathcal{U}, \one}(\pi)$ is the ordinary restriction of representations.
		\end{enumerate}
	 	
\end{remark}

\subsection{Symplectic group and its maximal parabolic subgroups}\label{maximalparabolicsinSp}

For a finite dimensional non-degenerate symplectic vector space $(V,B)$ over a field, let $Sp(V)$ denote the subgroup of $GL(V)$ preserving the alternating form $B.$ 

We refer the reader to \cite[\S 8.1]{PR} and \cite[\S 3]{TA} for comparison with this section. Let $V$ denote the $2n$-dimensional $F$-vector space $F^{2n}.$ Let $B_J$ 
denote the alternating non-degenerate bilinear form on $V$ given by the 
$2n\times 2n$ alternating matrix $J=\left[\begin{array}{cc}
	0&I_{n}  \\
	-I_{n}&0 
\end{array}
\right],$ i.e., 
$B_J(v,w)=~^tv J w$ for all $v,w\in V.$ Let $Sp_{2n}(F):=\{g\in GL_{2n}(F): ~^tg J g=J \}$, the subgroup of $GL_{2n}(F)$ preserving the alternating form $B_J.$  For $A, B, C, D \in M_n(F),$ put the condition
\begin{equation}\label{MatrixDef}
	~^tAC = ~^tCA, ~^tBD = ~^tDB, ~ ^tDA - ~^tBC = I_n.
\end{equation}  
One then has, $Sp_{2n}(F) = \left\{ \left[\begin{array}{cc}
		A & B \\
		C & D
	\end{array}\right]\in GL_{2n}(F) : A, B, C, D \in M_n(F) \text{ such that } \eqref{MatrixDef}  \text{ holds} \right\}.$

We take $Sp_{0}(F)$ to be the trivial group. Let $\{e_1,\dots, e_{2n}\}$ denote the standard basis of $V.$  We recall that a subspace $Y$ of $V$ is said to be {\it totally isotropic} if $B_J(v,w)=0$ for all $v,w\in Y.$ It is a standard fact that any maximal totally 
isotropic subspace of $V$ is of dimension $n.$ Let $V_{0}$ denote the maximal 
totally isotropic subspace $\langle e_1,\dots, e_n \rangle$ of $V.$  The dual space of $V_{0}$ is  $V_{0}^{\vee}=\langle e_{n+1}, 
	\dots, e_{2n}\rangle.$ In particular, the dual basis $\{{e_1^{\vee}, \dots, 
		e_n^{\vee}}\}$ of the basis $\{e_1,\dots, e_n\}$ of $V_{0}$ is given by  
	$e_i^{\vee}=e_{n+i}$ for $1\leq i \leq n.$ With this identification of the dual space $V_{0}^{\vee},$ 
	one has the direct sum  $V=\langle e_1,\dots, e_n \rangle  \oplus \langle 
	e_1^{\vee},\dots, e_n^{\vee} \rangle =V_{0}\oplus V_{0}^{\vee}.$ 
	
	Fix a positive integer $k$ such that $1 \leq k \leq n$ and let $\bm{\Lambda}(k)$ denote the set of all $k$-dimensional totally isotropic subspaces of $V.$ The natural action of the symplectic group $Sp_{2n}(F)$ on $\bm{\Lambda}(k)$ given by $g \cdot X = g(X)$ for $g \in Sp_{2n}(F)$ and $X \in \bm{\Lambda}(k)$ is transitive. For a fixed $k\in \{1,\dots,n\},$ consider the totally isotropic subspace $ X_0 := \langle e_1, \dots, e_k \rangle$ in $\bm{\Lambda}(k).$  The stabilizer $P_{k}$ of $X_0$ in $Sp_{2n}(F)$ is a maximal parabolic subgroup of $Sp_{2n}(F).$ The natural action of $Sp_{2n}(F)$ on the space $P_k \backslash Sp_{2n}(F)$ defined by $g \cdot P_kh = P_kh g^{-1}$ where $g, h \in Sp_{2n}(F)$ is transitive and we identify  $P_k \backslash Sp_{2n}(F)$ and $\bm{\Lambda}(k)$ as 
$Sp_{2n}(F)$-spaces via the map $P_{k} g \mapsto g^{-1}(X_0).$

Fix $1\leq k\leq n.$ If $k<n,$ we have the Levi decomposition of $P_k$ given by
$P_k = M_k N_k$ where the Levi subgroup $M_k$ is isomorphic to $GL_{k}(F) \times Sp_{2n-2k}(F)$ which is embedded in $Sp_{2n}(F)$ via
\begin{equation}\label{embedingofleviSp_{2n}(F)}
\left(g, \left[\begin{array}{cc}
a & b \\
c & d
\end{array}\right] \right) \mapsto \left[\begin{array}{cccc}
g &  &  &  \\
& a &  & b \\
&  & ~^tg^{-1} &  \\
& c &  & d 
\end{array}\right].
\end{equation}
The unipotent radical $N_k$ of $P_k$ consists of matrices of the form 
$  \left[\begin{array}{cccc}
I_k & * & * & * \\
& I_{n-k} & * &  \\
&  & I_k &  \\
&  & * & I_{n-k} 
\end{array}\right].$ The precise relation among the blocks are described explicitly in \S \ref{Struture-of-N_k}. If $k=n,$ $P_n$ has the Levi decomposition $P_n=M_nN_n,$ where the Levi subgroup $M_n$ is isomorphic to $GL_n(F)$ which is embedded in $Sp_{2n}(F)$ via 
\begin{equation}\label{embedingofleviSp_{2n}(F)-Siegel}
g \mapsto \left[\begin{array}{cc}
g &  \\
& ~^tg^{-1}  \\
\end{array}\right].\end{equation}
The unipotent radical $N_n$ in this case is isomorphic to $SM_n(F).$

\subsubsection{Structure of $P_{k}$}

Suppose  $ a, b, c, d \in M_{n-k}(F), g\in GL_k(F), h, y \in M_k(F),  u, w \in M_{n-k,k}(F),$ $x, z, v, m \in M_{k, n-k}(F).$ We put two conditions as follows:

\begin{equation}\label{ConditionOnPk1}
\left[\begin{array}{cc}
a & b \\
c & d
\end{array}\right] \in Sp_{2n-2k}(F) 
\end{equation}
and
\begin{equation}\label{ConditionOnPk2}
g^{-1}x + ~^tua - ~^twc = 0, g^{-1}z + ~^tub - ~^twd = 0, g^{-1}y + ~^tu w \in SM_k(F). 
\end{equation}
Suppose $\sigma = \left[\begin{array}{cccc}
g & x & y & z \\
& a & w & b \\
& v & h & m \\
&  c& u & d 
\end{array}\right] \in GL_{2n}(F).$ Note that $\sigma$ stabilizes $X_0.$ Moreover, $\sigma \in Sp_{2n}(F)$ if and only if the entries of $\sigma$ satisfy $~^tgh = I_k, v =0=m,$ \eqref{ConditionOnPk1} and \eqref{ConditionOnPk2}.
We then have,
$$P_{k} = \left\{\left[\begin{array}{cccc}
g & x & y & z \\
& a & w & b \\
&  & ^tg^{-1} &  \\
&  c& u & d 
\end{array}\right] \in GL_{2n}(F) : \begin{array}{llll}
g \in GL_{k}(F), \\
y \in M_k(F), \\
x, z \in M_{k, n-k}(F), \\
u, w \in M_{n-k,k}(F), \\
a, b, c, d \in M_{n-k}(F)
\end{array} \text{ such that } \eqref{ConditionOnPk1} \text{ and } \eqref{ConditionOnPk2} \text{ hold}  \right\}.$$

\subsubsection{Structure of $N_{k}$}\label{Struture-of-N_k}

Suppose  $ x, z \in M_{k, n-k}(F), y \in M_k(F)$ and $ u, w \in M_{n-k, k}(F).$ Assume $\sigma =\left[\begin{array}{cccc}
I_k & x & y & z \\
& I_{n-k} & w &  \\
&  & I_k &  \\
&  & u & I_{n-k} 
\end{array}\right] \in GL_{2n}(F).$ Then, $\sigma \in Sp_{2n}(F)$ if and only if $w = ~^tz, u = -^tx$ and $y - x ~^tz \in SM_k(F).$ Then,
$$N_{k} = \left\{\left[\begin{array}{cccc}
I_k & x & y & z \\
& I_{n-k} & ^tz &  \\
&  & I_k &  \\
&  & -^tx & I_{n-k} 
\end{array}\right] \in GL_{2n}(F): \begin{array}{lll}
x \in M_{k, n-k}(F),\\
y \in M_k(F), \\
z \in M_{k, n-k}(F)
\end{array} \text{ such that } y - x ~^tz \in SM_k(F) \right\}.$$

\subsubsection{The Siegel parabolic subgroup}\label{SiegelparabolicSp_{2n}(F)}

If $k=n,$ we shall denote $P_{n}$ simply by $P.$ The maximal parabolic subgroup $P$ is called the {\it{Siegel parabolic}} subgroup of $Sp_{2n}(F).$ We shall also denote $M_{n}$ and $N_{n}$ simply by $M$ and $N$ respectively. Explicitly, 
\begin{equation*} P=\left\{\left[ \begin{array}{cc}
g&X  \\
0& ^tg^{-1} 
\end{array}
\right]: g\in GL_{n}(F), X\in M_n(F) \text{ and } g^{-1}X \in SM_n(F) \right\},
\end{equation*}
\begin{equation*}
M= \left\{\left[ \begin{array}{cc}
g& 0  \\
0& ^tg^{-1} \
\end{array}
\right]: g\in GL_{n}(F)\right\} \cong GL_n(F)
\end{equation*}
and
\begin{equation*}N=\left\{\left[ \begin{array}{cc}
I_{n}&X  \\
0& I_{n} \
\end{array}
\right]: X\in SM_n(F)\right\} \cong SM_n(F).\end{equation*}

\subsubsection{The Klingen parabolic subgroup}

If $k=1,$ we shall denote $P_{1}$  by $Q, M_1$ by $L$ and $N_1$ by $U.$  The maximal parabolic subgroup $Q$ is called the {\it{Klingen parabolic}} subgroup of $Sp_{2n}(F).$ To describe the elements of $Q,$ suppose $t \in F^{\times}, x,z \in M_{1, n-1}(F),u,w \in M_{n-1,1}(F)$ and $a,b,c,d \in M_{n-1}(F).$ 
Put the following condition:

\begin{equation}\label{ConditionOnKlingenParabolicforSp_2n}
t^{-1}x+~^tua = ~^twc, t^{-1}z+~^tub = ~^twd \text{ and } \left[\begin{array}{cc}
a & b \\
c & d 
\end{array}\right] \in Sp_{2n-2}(F).
\end{equation}
Then,
\begin{equation*}
Q = \left\{\left[\begin{array}{cccc}
t & x & y & z \\
& a & w & b \\
&  & t^{-1} &  \\
&  c& u & d 
\end{array}\right] : \begin{array}{ll}
t \in F^{\times}, y \in F,x,z \in M_{1, n-1}(F),\\
u,w \in M_{n-1,1}(F),a,b,c,d \in M_{n-1}(F)
\end{array} \text{ such that }\eqref{ConditionOnKlingenParabolicforSp_2n}  \text{ holds}\right\},
\end{equation*}

\begin{equation*}
L = \left\{\left[\begin{array}{cccc}
t &  &  &  \\
& a &  & b \\
&  & t^{-1} &  \\
& c &  & d 
\end{array}\right] : t \in F^{\times}, \left[\begin{array}{cc}
a & b \\
c & d 
\end{array}\right] \in Sp_{2n-2}(F)\right\} \cong F^{\times} \times Sp_{2n-2}(F)
\end{equation*} 
and
\begin{equation*}
U = \left\{\left[\begin{array}{cccc}
1 & x & y & z \\
& 1 & ^tz &  \\
&  & 1 &  \\
&  & -^tx & 1 
\end{array}\right] : y \in F,x,z \in M_{1, n-1}(F) \right\}.
\end{equation*}

For any abelian $\ell$-group $G,$ we denote by $\widehat{G}$ the group of characters 
of $G$ with values in $\C^{\times}.$ Since $P$ normalizes $N,$ $P$ acts on $N$ and 
hence on the group $\widehat{N}$ naturally via $p\cdot  \psi= \psi^p$ where 
$\psi^{p}(n)=\psi(p^{-1}np)$ for all $p\in P,$ $\psi\in \widehat{N}$ and $n\in N.$ Let $S_{\psi}$ denote the subgroup 
of $P$ stabilizing the character $\psi$ of $N$ under this action. Put $M_{\psi}=S_{\psi} \cap M.$ Then, $N$ is a normal subgroup of $S_{\psi}$ and we have a semi-direct product 
$S_{\psi}=M_{\psi}N.$ Associated to any smooth representation $(\pi, V)$ of $Sp_{2n}(F)$ and a character $\psi$ of $N,$ the  twisted Jacquet module $r_{N,\psi}(\pi)$ is an $M_{\psi}$-module.

\section{Preliminary Results}\label{prelims}

We present some general results that we shall use in the article. In Section \ref{Linearalgebralemmas}, we recall some results from symplectic linear algebra. 
Even though we need the results only for $n=2,$ we have chosen to present the statements in general for the symplectic space $F^{2n}.$
The groups associated to a character of the unipotent radical $N$ of $P$ are described in Section \ref{Spsi-defn-etc}. The Geometric lemma of \cite{BZ2} applicable to our calculation is recalled in Section \ref{geometriclemma}. In Section \ref{setuplemma}, we record a generality so as to apply the geometric lemma later.

\subsection{Some general results}\label{Linearalgebralemmas}

Let $P$ be the Siegel parabolic subgroup of $Sp_{2n}(F).$ Recall from \S \ref{maximalparabolicsinSp} that  $P_k \backslash Sp_{2n}(F)$ is identified with $\bm{\Lambda}(k)$ by $P_{k} g \mapsto g^{-1}(X_0)$ where $g \in Sp_{2n}(F).$ The following lemma gives a complete set of orbit representatives for the action of $P$ on $P_k \backslash Sp_{2n}(F).$

\begin{lemma}\label{lemmafororbitsGeneral}
	
	Fix a positive integer $k$ such that $1 \leq k \leq n$ and let $\bm{\Lambda}(k)$ denote the set of all $k$-dimensional totally isotropic subspaces of the $2n$-dimensional symplectic $F$-vector space $V.$ Define $k+1$ totally isotropic subspaces in $\bm{\Lambda}(k)$ as follows:
	\begin{align*}
	X_{0}&=\langle e_1,\dots, e_k \rangle,\\
	X_j&=\langle e_1^{\vee}, \dots, e_{j}^{\vee}, e_{j+1}, \dots, e_k \rangle, \text{ for } 1\leq j\leq k-1, \text{ and }\\
	X_{k}&=\langle e_1^{\vee}, \dots, e_k^{\vee} \rangle.
	\end{align*}
	
	For the action of $P$ on $\bm{\Lambda}(k),$ there are exactly $k+1$  orbits whose representatives can be taken to be one of the points $X_{0},X_1,\dots, X_{k}.$ 
\end{lemma}

We will skip the proof of the lemma as it  is a standard application of Witt's theorem extending isometries. The point of the lemma is that if $Y\in \bm\Lambda(k),$ $Y\in Orb_P(X_j)$ if and only if $\dim(Y\cap V_0)=k-j.$ As a consequence to Lemma \ref{lemmafororbitsGeneral}, we have the following proposition which gives a complete set of $(P_k, P)$-double coset representatives in $Sp_{2n}(F).$

\begin{proposition}\label{PdoublecosetGeneral}
	Fix a positive integer $k$ such that $1 \leq k < n.$ Put $w_0=I_{2n},$ 
	$$w_j=\left[\begin{array}{cccccc}
	&  &  & -I_j & & \\
	& I_{k-j} &  &  & &\\
	&  & I_{n-k} & & & \\
	I_j &  &  &  & &\\  
	&  &  &  &I_{k-j}&\\  
	&  &  &  & &I_{n-k}  
	\end{array}\right], \text{ for }1 \leq j \leq k-1,$$ and $w_k=\left[\begin{array}{cccc}
	&  & -I_k &  \\
	& I_{n-k} &  &  \\
	I_k &  &  &  \\ 
	&  &  & I_{n-k}  
	\end{array}\right].$ Then, $\{{w_0}^{-1},\dots, {w_k}^{-1}\}$ is a complete set of $(P_k, P)$-double coset representatives in $Sp_{2n}(F).$ 
\end{proposition}

\begin{proof}
	The proof follows from Lemma \ref{lemmafororbitsGeneral} as it is easy to verify that $w_j(X_{0}) = X_j$ for every $j \in \left\{0, 1, \dots, k\right\}.$
\end{proof}

The following proposition is also a consequence of Lemma \ref{lemmafororbitsGeneral} and  is the analogue of Proposition \ref{PdoublecosetGeneral} in the case where $k=n.$

\begin{proposition}\label{Pdoublecoset}
	Put $\sigma_0=I_{2n},$ $\sigma_n=\left[ \begin{array}{cc}
	& -I_{n}  \\
	I_{n}& 
	\end{array}
	\right]$ and $
	\sigma_j=\left[\begin{array}{cccc}
	&  & -I_j &  \\
	& I_{n-j} &  &  \\
	I_j &  &  & \\
	&  &  & I_{n-j} 
	\end{array}\right],$ for $1\leq j \leq n-1.$ Then, $\{ {\sigma_0}^{-1}, \dots, {\sigma_n}^{-1} \}$ is a complete set of $(P,P)$-double coset representatives of $Sp_{2n}(F).$
\end{proposition}

\begin{proof}
	
	In this case, $X_0= V_0$ and $X_j=\langle e_1^{\vee}, \dots, e_{j}^{\vee}, e_{j+1}, \dots, e_n \rangle,$ for $1\leq j\leq n.$ It is easy to verify that $\sigma_j(X_{0}) = X_j$ for every $j \in \left\{0, 1, \dots, n\right\}.$  
\end{proof}

\begin{remark}\label{open-orbits}
	\begin{enumerate}
		\item In the case where $n=k=2,$ the $P$-orbits on $P\backslash Sp_4(F)$ are given by $Orb_P(X_0),$ $Orb_P(X_1)$ and $Orb_P(X_2).$ Of these, $Orb_P(X_0)$ is a point which corresponds to the closed orbit and $Orb_P(X_2)$ corresponds to the open orbit.
		
		\item In the case where $n=2$ and $k=1,$ the $P$-orbits on $Q\backslash Sp_4(F)$ are $Orb_P(X_0)$ which is closed and $Orb_P(X_1)$ which is open.
	\end{enumerate}	
\end{remark}

\subsection{Stabilizers of the characters of the unipotent radical $N$ of $P$}\label{Spsi-defn-etc}

We fix a non-trivial additive character $\psi_{_0}: F\to \C^{\times}$ throughout 
this article. For the Siegel parabolic subgroup $P,$ its unipotent radical $N$ can be identified with 
$SM_n(F)$ via \S \ref{SiegelparabolicSp_{2n}(F)} and therefore the dual 
group $\widehat{N}$ is isomorphic to $SM_n(F),$ since $SM_n(F)$ is self-dual. Given 
any $A\in SM_n(F),$ we can define a character $\psi_{_A}:N\to \C^{\times }$ by 
\begin{equation*}\label{psiAdefn}\psi_{_A}\left( \left[ \begin{array}{cc}
I_{n}&X  \\
0& I_{n} \
\end{array}
\right]\right)=\psi_{_0}(\tr(AX))\end{equation*}
for each $X\in SM_n(F).$ The character $\psi_{_A}$ is said to be {\it non-degenerate} if $A$ is invertible and  {\it degenerate} otherwise. 

Let $W$ denote the $m$-dimensional $F$-vector space $F^{m}.$ Given any $A\in 
SM_m(F),$ let $B_{A}$ denote the symmetric bilinear form on $W$ defined by 
$B_A(x,y)=~^txAy$ for all $x,y\in W$ and let $O_{m}(F,A)$ denote the orthogonal 
subgroup of $GL_m(F)$ preserving the symmetric bilinear form defined by $B_A.$ 
Explicitly, $O_{m}(F,A)=\{g\in GL_m(F): ~^tg A g=A \}$ (see for instance \cite[Ch. 1, \S 2]{Scharlau}). With this notation, the explicit description of the elements of $S_{_{\psi_{_A}}}$ and $M_{\psi_{_A}}$ are noted in the following lemma.

\begin{lemma}\label{SpsiMpsi}
	Fix a non-zero matrix  $A \in SM_n(F).$ For the character $\psi_{_A}$ of $N$ we have,
	\begin{enumerate}
		\item 	$S_{\psi_{_A}} := Stab_P({\psi_{_A}}) = \left\{ \left[\begin{array}{cc}
		g & X \\
		0 & ^tg^{-1}
		\end{array}\right] \in P :g \in GL_n(F), X \in M_n(F) \text{ and }~ ^tg A g = A\right\}.$

		\item $M_{\psi_{_A}} := M \cap S_{\psi_{_A}} = \left\{ \left[\begin{array}{cc}
		g & 0 \\
		0 & ^tg^{-1}
		\end{array}\right] \in M :g \in GL_n(F) \text{ and }~ ^tg A g = A\right\}\cong  O_{n}(F, A).$

		\item $N$ is normal in $S_{\psi_{_A}}$ and $S_{\psi_{_A}}$ is a semi-direct product of $M_{\psi_{_A}}$ with $N.$ 
	\end{enumerate}
\end{lemma}	

\begin{proof}
	For proof of (1), fix $\sigma:= \left[\begin{array}{cc}
	g & X \\
	0 & ^tg^{-1}
	\end{array}\right] \in P.$ Then, $\sigma\in S_{\psi_{_A}}$ if and only if     
	$$\psi_{_0}\left(\tr[(^tg^{-1} A g^{-1} - A)X']\right) = 1,$$ for all $X' \in SM_n(F).$ We observe that for any non-zero $d\in SM_n(F),$ the trace form $\tr_{d}: SM_n(F) \rightarrow F$ defined by $\tr_{d}(w) = \tr(dw)$ for  $w \in SM_n(F)$ is a surjective map. 
	Thus, $\sigma\in S_{\psi_{_A}} $  if and only if $~^tg^{-1} A g^{-1} - A =0.$ This proves (1). 
	The statement (2) is obvious and (3) is straightforward for which we skip the proof. 
\end{proof}

\subsection{Geometric Lemma}\label{geometriclemma}

In this paragraph, we recall the Bernstein-Zelevinsky Geometric lemma for $Sp_{4}(F)$ in our context. In the rest of this subsection, we fix $n=2$ and let $k\in \{1,2\}.$  Then, $P_1=Q,$ $P_2=P$ and we have $M_1=L,N_1=U,$ $M_2=M$ and $N_2=N.$ We have $P_k=M_kN_k$ for $k\in \{1,2\}.$ Let $\psi$  denote the character $\psi_{_C}$ of $N$ defined by \eqref{CharacterCorrspondngtoC}. Recall that $S_{\psi}=M_{\psi}N.$  Note that $i_{_{N_k,\one}}(\rho)=\ind_{P_k}^{Sp_4(F)}(\rho)$ for any $\rho\in \Alg(M_k).$ The Geometric Lemma \cite[\S 5, Theorem 5.2]{BZ2} of Bernstein-Zelevinsky gives a recipe to obtain the general structure of twisted Jacquet module $r_{N,\psi}(\ind_{P_k}^{Sp_4(F)}(\rho))$ for any $\rho\in \Alg(M_k)$ under certain hypothesis (see \cite[\S 5, 5.1(1)--(4)]{BZ2}). 

Consider the parabolically induced representation $\ind_{P_k}^{Sp_{4}(F)}(\rho)$ where $\rho\in \Alg(M_k).$ 
In this scenario, it is easy to see that the hypothesis (1) and (2) of \cite[\S 5, 5.1(1)--(4)]{BZ2} mentioned above holds. To check hypothesis (3), one needs to show that $S_{\psi}$ has finitely many orbits while acting on the quotient $P_k \backslash Sp_{4}(F).$ The hypothesis (4) involves a technical criterion. We recall the relevant definition from \cite[\S 5, 5.1(3)]{BZ2}. Call a subgroup $H\subset Sp_4(F)$ decomposable with respect to a pair of subgroups  $(H_1, H_2),$ if $H \cap  (H_1 H_2) = (H \cap H_1)\cdot (H \cap H_2).$

Consider the action of $Sp_{4}(F)$ on $P_k \backslash Sp_{4}(F)$ defined by $g \cdot P_kh = P_kh g^{-1}$ where $g, h \in Sp_{4}(F).$  We restrict this action to  $S_{\psi}$  and prove that there are only finitely many orbits  for this restricted action. This verifies that the hypothesis (3) of \cite[\S 5.1]{BZ2} holds.  Let $\{Z_j:1\leq j\leq  m\}$ be the distinct $S_{\psi}$-orbits. For each $S_{\psi}$-orbit $Z_j$ we find an element $w_j\in Sp_{4}(F)$ such that $w_j\cdot P_k=P_kw_j^{-1}\in Z_j.$  Then, $Sp_{4}(F)=\bigsqcup\limits_{j=1}^{m}  Z_j= \bigsqcup\limits_{j=1}^{m} Orb_{S_{\psi}} (P_kw_j^{-1})= \bigsqcup\limits_{j=1}^{m} P_kw_j^{-1} S_{\psi}.$ If we put $\mathcal{W}:=\{w_j: 1\leq j \leq m\},$ then  $\{w^{-1}: w\in \mathcal{W}\}$ forms a complete set of $(P_k,S_{\psi})$-double coset representatives in $Sp_{4}(F).$ Thus, $\mathcal{W}=\{w: w^{-1}\in P_k \backslash Sp_{4}(F)/S_{\psi}\}.$

For an element $g$ and a subgroup $H$ of $Sp_{4}(F),$ $g(H)$ shall denote the conjugate $gHg^{-1}$ of $H.$ To check hypothesis (4) in \cite[\S 5.1]{BZ2},  we have to verify for each $w \in \mathcal{W}$  that the following set of conditions hold:

\begin{enumerate}
	\item[(a)] the groups $w(P_k), w(M_k)$ and $w(N_k)$ are decomposable with respect to the pair of subgroups $(M_{\psi}, N)$;
	
	\item[(b)] the groups $w^{-1}(S_{\psi}), w^{-1}(M_{\psi})$ and $w^{-1}(N)$ are decomposable with respect to the pair of subgroups $(M_k, N_k).$
\end{enumerate} 

If all the conditions (1)--(4) of \cite[\S 5.1]{BZ2} hold, to determine whether there is a contribution to the twisted Jacquet module from the double coset containing $w^{-1}$ we have to verify whether the following condition holds:
\begin{equation}\label{starcondition}
\psi=1  \text{ on } wN_kw^{-1}\cap N. 
\end{equation}  

For $w \in \mathcal{W},$ one defines the functor $\Phi_w:\Alg(M_k)\to \Alg(M_\psi)$ as follows. If \eqref{starcondition} does not hold for a particular $w \in \mathcal{W},$ then the corresponding double coset does not contribute to the twisted Jacquet module, i.e., the contribution is zero and we set $\Phi_w = 0.$
On the other hand, if \eqref{starcondition} holds for a particular $w\in \mathcal{W},$ then we define the following:

\begin{enumerate}
	\item ${M_\psi}'(w)=wM_kw^{-1}\cap M_{\psi}.$
	
	\item ${N_k}'(w)= wN_kw^{-1} \cap M_{\psi}.$
	
	\item $N'_{w}=w^{-1}Nw \cap M_k.$
	
	\item $\psi'_{w}=(\psi^{w^{-1}})_{|_{N'_{w}}}.$
	
	\item $M_k'(w)=w^{-1}M_{\psi}w \cap M_k.$
\end{enumerate}

One also has the following functors defined:
\begin{enumerate}
	\item[(i)] $ r_{N'_{w},\psi'_{w}}: \Alg(M_k)\to \Alg(M_k'(w)).$
	
	\item[(ii)] $w: \Alg(M_k'(w)) \to\Alg({M_\psi}'(w))$ i.e., for each $\rho \in \Alg(M_k'(w))$ we define $\rho^{w}(m)=\rho(w^{-1}mw)$ for $m\in {M_{\psi}}'(w).$ Then, $\rho^{w}\in \Alg({M_{\psi}}'(w)).$ 
	
	\item[(iii)] $i_{_{{N_k}'(w), \one}}:\Alg({M_\psi}'(w))\to \Alg(M_{\psi}).$ 
\end{enumerate}

We have two characters $\varepsilon_1(w)$ and $\varepsilon_2(w) $ defined as follows:

$\varepsilon_1(w): = {\rm mod}_{N_k}^{\frac{1}{2}} \cdot {\rm mod}_{w^{-1}S_{\psi}w \cap N_k}^{-\frac{1}{2}}$ is a character of the group $M_k'(w)$ and 

$\varepsilon_2(w): = {\rm mod}_{N}^{\frac{1}{2}} \cdot {\rm mod}_{w P_k w^{-1} \cap N}^{-\frac{1}{2}}$ is a character of the group ${M_\psi}'(w).$ Define the character $\varepsilon(w):M_k'(w) \to \C^{\times}$ by $\varepsilon(w)[m]= \varepsilon_1(w)[m] \cdot \varepsilon_2(w)[wmw^{-1}]$ for each $m\in M_{k}'(w).$ 
Put $\Phi_w=i_{_{{N_k}'(w),\one}}\circ w \circ \varepsilon(w)\circ r_{N'_{w},\psi'_{w}}.$ The contribution from the orbit $P_k w^{-1}S_\psi$ to the twisted Jacquet module $r_{N,\psi}(\ind_{P_k}^{Sp_4(F)}(\rho))$ is  given by the induced representation of the group $M_{\psi},$ denoted by $\Phi_w(\rho)$ where
\begin{equation}\label{contributionPhi(w)}
\Phi_w(\rho)=i_{_{{N_k}'(w), \one}}([\varepsilon(w) \cdot r_{N'_{w},\psi'_{w}}(\rho)]^{w}).
\end{equation}

The action of the representation $[\varepsilon(w) \cdot r_{N'_{w},\psi'_{w}}(\rho)]^{w}$ on an element $m\in {M_\psi}'(w)$ is given by 
\begin{equation}\label{inducingaction}
[\varepsilon(w) \cdot r_{N'_{w},\psi'_{w}}(\rho)]^{w}(m)= \varepsilon(w)[w^{-1}mw]\cdot r_{N'_{w},\psi'_{w}}(\rho)[w^{-1}mw].
\end{equation}

The first statement of the Geometric lemma (\cite[\S 5, Theorem 5.2]{BZ2}) is then that $r_{N,\psi}(\ind_{P_k}^{Sp_{4}(F)}(\rho))$ is glued from the representations $\Phi_w(\rho)$ where $w$ varies over $\mathcal{W}.$

\subsection{Set up for the Geometric Lemma}\label{setuplemma}

We note the following generality. Suppose that $G$ is an $\ell$-group and $H,K$ are closed subgroups of $G$ such that $K$ is a closed subgroup of $H.$ Suppose $G$ acts on an $\ell$-space $X$ transitively and $H$ acts on $X$ with $k+1$ orbits $Orb_{H}(x_j)$ for $j=0,1,\dots, k.$  Suppose the $K$-action on each $Orb_H(x_j)$ decomposes $Orb_H(x_j)$ into disjoint $K$-orbits as follows: 
\[Orb_H(x_j)=\bigsqcup_{i=1}^{r_j} Orb_K(x_{_{ji}}),\]
where $x_{_{j1}}=x_j$ for each $j\in \{0,1,\dots, k\}.$ 
Then, $X$ decomposes into a disjoint union of $K$-orbits as follows:
\begin{equation}\label{general-orbit-for-BZ}
X= \bigsqcup_{j=0}^{k}\left(\bigsqcup_{i=1}^{r_j} Orb_K(x_{_{ji}})\right).
\end{equation}

Suppose that $g_0=e, g_1,\dots, g_k$ are elements of $G$ such that $g_j\cdot x_0=x_j$ for $0 \leq j\leq k.$ Also, suppose that $h_{_{j1}}, \dots, h_{_{jr_j}}$ are elements of $H$ such that $h_{_{ji}}\cdot x_{j1}=x_{_{ji}}$ for $1\leq i\leq r_j$ and $0 \leq j\leq k.$ 
Then a complete set of orbit representatives for the action of $K$ on $X$ is given by 
$$ \{ h_{_{ji}} \cdot (g_{j}  \cdot x_{_{01}}): 1\leq i \leq r_{j},0\leq j \leq k\}.$$

Let $G_0$ be a closed subgroup of $G$ and put $X= G_0 \backslash G.$ We would like to apply the principle we have observed in the previous paragraph to $K$ acting on $X$ where the action of $K$ on $X$ is the restriction of the $G$-action on $X$ given by: for $g \in G$ and $G_0h \in G_0 \backslash G$ define $g \cdot G_0h = G_0h g^{-1}.$ We summarize the above discussion in the following lemma whose proof follows from \eqref{general-orbit-for-BZ}.

\begin{lemma}\label{hijgj} 
	Suppose $G,G_0,H,K$ are $\ell$-groups with $G_0,H,K$ closed subgroups of $G$ and $K$ is a closed subgroup of $H.$ Assume the following:
	\begin{enumerate}
		\item $G_0 g_0,\dots, G_0 g_k$ are a complete set of disjoint $H$-orbits for the action of $H$ on $G_0 \backslash G$ with $g_0=e.$
		
		\item For each $j\in \{0,1,\dots, k\}, \{ G_0g_j{h_{ji}}^{-1} K: 1\leq i \leq r_j\}$ is a complete set of disjoint $K$-orbits for the action of $K$ on $G_0 g_j.$
	\end{enumerate} 
	
	Then, $\{ g_{j}{h_{_{ji}}}^{-1}: 1\leq i \leq r_{j},0\leq j \leq k\}$ is a complete set of 
	$(G_0, K)$-double coset representatives for $G,$ i.e., $G=\bigsqcup\limits_{j=0}^{k}\left( \bigsqcup\limits_{i=1}^{r_j} G_0 g_j{h_{ji}}^{-1} K \right).$
\end{lemma}

\section{Orbits for the $S_{\psi}$-action}\label{Orbits-Spsi-action}

Recall that we have fixed $C = \left[\begin{array}{cc}
	\gamma &0 \\
	0& 0
\end{array}\right]$ where $\gamma\in F^{\times}.$ We shall denote the character $\psi_{_C}$ of \eqref{CharacterCorrspondngtoC} by $\psi.$ By \S \ref{Spsi-defn-etc} and \eqref{O2(F,C)-defn}, we have
$O_{2}(F,C) =
\left\{\left[\begin{array}{cc}
\pm 1 & 0 \\
y & d
\end{array}\right]: d \in F^{\times}, y \in F \right\}.$ In view of Lemma \ref{SpsiMpsi}, the subgroups $S_{\psi}$ (resp. $M_{\psi}$) which stabilizes the character $\psi$  in $P$ (resp. $M$) are given by
\begin{enumerate}
\item[(a)] $S_{\psi} = \left\{ \left[\begin{array}{cc}
	g & X \\
	0 & ^tg^{-1}
\end{array}\right] \in P : g \in O_{2}(F,C)  \text{ and } X \in M_2(F) \right\},$ and

\item[(b)] $M_{\psi} = \left\{ \left[\begin{array}{cc}
	g & 0 \\
	0 & ^tg^{-1}
\end{array}\right] :g \in O_{2}(F,C) \right\} \cong O_{2}(F,C).$
\end{enumerate}

We note that the group $S_{\psi}$ and hence $M_{\psi}$ do not depend on the element $\gamma\in F^{\times}.$
Also, $M_{\psi}$ is the image of $O_{2}(F,C)$ under the embedding \eqref{embedingofleviSp_{2n}(F)-Siegel} of $GL_{2}(F)$ in $Sp_{4}(F).$ 
With the above notations in place, we study the restriction of a principal series representation induced from a maximal parabolic subgroup of $Sp_{2n}(F)$ to its Siegel parabolic subgroup $P$ in \S \ref{restriction-to-P}. In \S \ref{Action-of-Orthogonal-group}, we prove a preparatory result which is used in \S \ref{Action-of-Spsi-on-Sp_4modP} and \S \ref{Action-of-Spsi-on-Sp4modQ}, where we compute orbits for the action of the group $S_{\psi}$ on $P \backslash Sp_{4}(F)$ and $Q \backslash Sp_{4}(F)$ (via Propositions \ref{S-psi-orbit-Sp4/P} and \ref{S-psi-orbit-Sp4/Q}).  This will give us a complete set of $(P, S_{\psi})$-double coset (resp. $(Q, S_{\psi})$) representatives in $Sp_{4}(F).$  In \S \ref{KeySubgroups}, we calculate the modular characters of certain groups which shall be used later while proving the main results. We verify the decomposability conditions via Lemmas \ref{Decomposability-verification-Siegel} and \ref{Decomposability-verification-Klingen} in \S \ref{Decomposability}.

\subsection{The restriction to $P$ of a principal series representation of $Sp_{2n}(F)$}
\label{restriction-to-P}

Let $1\leq k \leq n.$ 
Let $\rho$ be a smooth representation of $M_k.$ We determine the restriction to $P$ of a principal series representation of $Sp_{2n}(F)$ of the form $\ind_{P_k}^{Sp_{2n}(F)}(\rho).$  First, we deal with induction from $P$ and then consider the case of induction from $P_k$ for $1\leq k<n.$

\begin{proposition}\label{Stabilizerfororbits}
	For $0 \leq j \leq n,$ put  $H_j= \sigma_j P{\sigma_j}^{-1} \cap P$ where $\sigma_j$'s are as in Proposition \ref{Pdoublecoset}.  Then,   
	\begin{enumerate}
		\item $H_0  = P,$
		
		\item 	$H_j = \left\{ \left[\begin{array}{cc}
		g & X \\
		0 & ^tg^{-1}
		\end{array} \right] \in P : g \in \overline{P_{{j, n-j}}}, X = \left[\begin{array}{cc}
		0 & y \\
		z  & w 
		\end{array} \right] \right\}$  for $1 \leq j \leq n-1,$ where $y \in M_{j, n-j}(F),$  $z \in M_{n-j, j}(F), w \in M_{n-j}(F),$ and
		
		\item $H_n  = M.$
		
	\end{enumerate}
	 
Moreover, for any $\rho\in \Alg(M),$ put $\pi=\ind_P^{Sp_{2n}(F)}(\rho).$ We have a filtration of $P$-modules given by
$$ \{0\} = \pi_{n+1} \subset \pi_n \subset \cdots \subset \pi_0=\pi_{_{|_P}} $$ 
 where $\pi_j/\pi_{j+1}\cong i_{H_j}^P([ \delta_P^{\frac{1}{2}} \cdot \rho ]^{\sigma_j})$ for $0\leq j\leq n.$
\end{proposition}

\begin{proposition}\label{StabilizerfororbitsGeneral}
	Fix $1 \leq k < n.$ For $j \in \{0, 1, \dots, k \},$ put $D_j= w_jP_k{w_j}^{-1}\cap P$ where $w_j$'s are as in Proposition \ref{PdoublecosetGeneral}. Then,
	\begin{enumerate}
		\item $D_0  = \left\{ \left[\begin{array}{cc}
		g & X \\
		0 & ^tg^{-1}
		\end{array} \right] \in P : g \in P_{k, n-k}, X \in M_n(F) \right\},$

		\item 
		$D_j =  \left\{ \left[\begin{array}{cc}
		g & X \\
		0 & ^tg^{-1}
		\end{array} \right] \in P :~ ^tg^{-1} \in P_{j,n-j}, X =   \left[\begin{array}{ccc}
		0 &y&z    \\
		u& v &w  \\
		0 & m &t 
		\end{array} \right] \in M_n(F) \right\}$ for $1 \leq j \leq k-1,$
		where $y \in M_{j, k-j}(F),  z \in M_{j, n-k}(F), u \in M_{k-j,j}(F),  v \in M_{k-j}(F),  w \in M_{k-j, n-k}(F), m \in M_{n-k, k-j}(F), t \in M_{n-k}(F),$ and

		\item $D_k =  \left\{ \left[\begin{array}{cc}
		g & X \\
		0 & ^tg^{-1}
		\end{array} \right] \in P :~ ^tg^{-1} \in P_{k, n-k}, X = \left[\begin{array}{cc}
		0 &0   \\
		0 & w
		\end{array} \right] \in M_n(F), w \in M_{n-k}(F) \right\}.$ \end{enumerate}
	For $\varrho\in \Alg(M_k),$ put $\Pi=\ind_{P_k}^{Sp_{2n}(F)}(\varrho).$ We have a filtration of $P$-modules given by
	$$\{0\}=  \Pi_{k+1} \subset \Pi_k \subset \cdots \subset \Pi_0=\Pi_{_{|_P}} $$ 
	where $\Pi_j/\Pi_{j+1}\cong i_{D_j}^P([ \delta_{P_k}^{\frac{1}{2}} \cdot \varrho]^{w_j})$  for $0\leq j\leq k.$\end{proposition}

The statements in Propositions \ref{Stabilizerfororbits} and \ref{StabilizerfororbitsGeneral} concerning the restriction is a standard application of Mackey theory and Lemma \ref{lemmafororbitsGeneral}. For more comments, see for instance \cite[Ch. 3, \S 3.2.1]{Omer}.

\subsection{Action of $O_{2}(F,C)$}
\label{Action-of-Orthogonal-group}

Let $B=\left\{\left[ \begin{array}{cc}
a& x \\
0& b
\end{array}\right]: a,b\in F^{\times}, x\in F\right\}$ denote the standard Borel subgroup of $GL_{2}(F).$ Denote by $\overline{B}$ the subgroup opposite to $B$ in $GL_{2}(F).$ In this section, we study the action of the orthogonal group $O_{2}(F,C)$ on $\overline{B} \backslash GL_{2}(F)$ which arises naturally in the sequel.

\begin{proposition}\label{O2C-acts-GmodBbar}
$\overline{B} \backslash GL_{2}(F)/O_{2}(F,C)=
	\{ h_0, h_1 \},$ where  $h_0=I_2$ and $h_1=\left[\begin{array}{cc}
			0 &1 \\
			1& 0
		\end{array}\right].$

\end{proposition}

\begin{proof}
	Let $X$ denote the space of all one-dimensional subspaces of $F^2.$
	Let $F^2$ be regarded as a quadratic space associated with the matrix $C = \left[\begin{array}{cc}
		\gamma &0 \\
		0& 0
	\end{array}\right].$ Fix $L_0=\langle e_2\rangle$ and $L_1=\langle e_1 \rangle$ in $X.$   $GL_{2}(F)$ acts on $X$ transitively and we identify $\overline{B} \backslash GL_{2}(F)$ with the orbit of $L_0.$  Consider the action of $O_{2}(F,C)$ on $X.$ 
Note that $L_0$ and $L_1$ are in distinct $O_2(F,C)$-orbits. Suppose $L = \langle ae_1+be_2 \rangle \in X.$	If $a$ (resp. $b$) is zero then $L = L_0$ (resp. $L_1$). On the other hand, if $a$ and $b$ are both non-zero, then by choosing $g = \left[\begin{array}{cc}
		1 & 0 \\
		\frac{b}{a}  & 1
	\end{array}\right] \in O_{2}(F,C)$ we get $ g(L_1) = L.$ Thus, orbits of $L_0$ and $L_1$ are the only $O_{2}(F,C)$-orbits in $X.$ We observe that $h_0(L_0) = L_0$ and $h_1(L_0) = L_1$ proving the proposition. 
\end{proof}

\begin{remark}\label{KlingenCtrivialcoset}
	Note that, $h_1^{-1}=h_1.$ It is easy to see that Proposition \ref{O2C-acts-GmodBbar} is true with the same double coset representatives even if we replace $\overline{B}$ with $B.$ Explicitly, $B\backslash GL_{2}(F)/O_{2}(F,C) = \{ h_0, h_1 \},$ where $h_0$ and $h_1$ are as in Proposition \ref{O2C-acts-GmodBbar}. 
\end{remark}

\subsection{Orbits for the $S_{\psi}$-action on $P \backslash Sp_{4}(F)$}
\label{Action-of-Spsi-on-Sp_4modP}

In the case of $Sp_{4}(F),$ by Proposition \ref{Stabilizerfororbits} and  Proposition \ref{Pdoublecoset}, we have
\begin{equation*}
H_1 = \left\{\left[\begin{array}{cccc}
b & 0 & 0 & y \\
c & d &  dyb^{-1} & w \\
0 & 0 & b^{-1} & -c b^{-1} d^{-1} \\ 
0 & 0 & 0 & d^{-1}  
\end{array} \right] : b, d\in F^{\times}, c, y, w\in F  \right\},
\end{equation*}

\begin{equation}\label{sigma1-def}
\sigma_1 = \left[\begin{array}{cccc}
0&0  & -1 &0  \\
0& 1 & 0 & 0 \\
1 & 0 & 0&0 \\ 
0&0  &0  & 1  
\end{array} \right] \text{ and } \sigma_2=\left[ \begin{array}{cc}
& -I_{2}  \\
I_{2}& 
\end{array}
\right].
\end{equation}

Put \begin{equation}\label{tau1-def}
\tau_1 = \left[\begin{array}{cc}
h_1 & \\
& ^t{h_1}^{-1} 
\end{array}\right]
\end{equation} where $h_1$ is as in Proposition \ref{O2C-acts-GmodBbar}. Denote  by $\overline{B}_1$ the image of the subgroup $\overline{B}$ under \eqref{embedingofleviSp_{2n}(F)-Siegel} in $Sp_{4}(F).$ We have $$\overline{B}_1=\left\{\left[\begin{array}{cccc}
b & 0 &  &  \\
c & d &   &  \\
&  & b^{-1} & -c b^{-1} d^{-1} \\ 
&  & 0 & d^{-1}  
\end{array} \right] : b, d\in F^{\times}, c\in F  \right\}.$$
Similarly, write $B_1$ for the image of $B$ under \eqref{embedingofleviSp_{2n}(F)-Siegel} in $Sp_{4}(F).$ Then, $$B_1=\left\{\left[\begin{array}{cccc}
b & c &  &  \\
0 & d &   &  \\
&  & b^{-1} & 0 \\ 
&  & -b^{-1}c d^{-1} & d^{-1}  
\end{array} \right]: b,d\in F^{\times},c\in F\right\}.$$

\noindent We note the following lemma.

\begin{lemma}\label{Spsi-OrbitinSiegelDegA} 
	 $H_1\backslash P/ S_{\psi}  = \left\{I_4, \tau_1\right\}.$

\end{lemma}

\begin{proof}
	Consider the action of $S_{\psi}$ on $H_1 \backslash P.$ Put  $N^{1}= \left\{ \left[\begin{array}{cc}
	I_2 & \left[\begin{array}{cc}
	0& y \\
	y & z
	\end{array} \right] \\
	& I_2
	\end{array}\right] : y, z \in F\right\}.$ Note that $\overline{B}_1$ normalizes $N^{1}$ and  we can write $H_1 = \overline{B}_1 N^{1}.$   
	Consequently, $H_1 \backslash P/ S_{\psi}  = \overline{B}_1 N^{1} \backslash MN/ M_{\psi} N,$ and can be identified with $ \overline{B}_1 \backslash M/ M_{\psi}.$ This further can be identified with $\overline{B} \backslash GL_{2}(F)/O_{2}(F,C).$ The lemma now follows from Proposition \ref{O2C-acts-GmodBbar}. 
\end{proof}

\begin{proposition}\label{S-psi-orbit-Sp4/P}
	$P\backslash Sp_{4}(F)/S_{\psi}=\{I_4, {\sigma_1}^{-1}, (\tau_1\sigma_1)^{-1}, {\sigma_2}^{-1}\}.$
\end{proposition}

\begin{proof}
	We consider the action of $S_{\psi}$ on $P \backslash Sp_{4}(F).$ Put $V_0= \langle e_1, e_2 \rangle, V_1= \langle e_1^{\vee}, e_2 \rangle$ and $V_2= \langle e_1^{\vee}, e_2^{\vee} \rangle.$ By Proposition \ref{Pdoublecoset} (applied with $n=2$),  $\sigma_j(V_0)=V_j$ for $0\leq j \leq 2.$ By Lemma \ref{lemmafororbitsGeneral} (applied with $n=2$ and $k=2$), the orbits for the action of $P$ on $P \backslash Sp_{4}(F)$ are $Orb_P(V_0)$($=\{V_0\}$), $Orb_P(\sigma_1\cdot V_0)$ and $Orb_P(\sigma_2\cdot V_0).$ We shall consider the $S_{\psi}$-action on each of these $P$-orbits so as to apply Lemma \ref{hijgj}.  
	
	\textit{Case}(1): It is clear that $S_{\psi}$ stabilizes $V_0$ and hence $Orb_{P}(V_0)=\{V_0\}=Orb_{S_{\psi}}(V_0).$

	\textit{Case}(2): $Orb_{P}(V_2)=M \backslash P$ as $M$ is the stabilizer in $P$ of $V_2$ by Proposition \ref{Stabilizerfororbits} (applied with $n=2$). For any $p\in P,$ as we can write $p=mn$ with $m\in M$ and $n\in N,$ the double coset $Mp S_{\psi}=MmnS_{\psi}=MS_{\psi}.$ Thus, $M \backslash P/S_{\psi}=\{I_4\}$ and $Orb_{P}(V_2)=Orb_{S_{\psi}}(V_2)=Orb_{S_{\psi}}(\sigma_2\cdot V_0).$

	\textit{Case}(3): By Proposition \ref{Stabilizerfororbits} (applied with $n=2$)  we have $Orb_P(V_1)= H_1 \backslash P.$ By Lemma \ref{Spsi-OrbitinSiegelDegA}, we have $H_1\backslash P/ S_{\psi}=\{I_4,\tau_1\}.$ Thus, $Orb_P(V_1)= H_1 \backslash P=Orb_{S_{\psi}}(V_1) \bigsqcup Orb_{S_{\psi}}(\tau_1\cdot V_1)=Orb_{S_{\psi}}(\sigma_1\cdot V_0) \bigsqcup Orb_{S_{\psi}}(\tau_1\sigma_1\cdot V_0).$
	We now apply Lemma \ref{hijgj} with $X= P\backslash Sp_{4}(F), G=Sp_{4}(F),H=G_0=P$ and  $K=S_{\psi}$ to complete the proof.
\end{proof}

\subsection{Orbits for the $S_{\psi}$-action on $Q \backslash Sp_{4}(F)$}
\label{Action-of-Spsi-on-Sp4modQ}

By Proposition \ref{StabilizerfororbitsGeneral} (applied with $n=2$ and $k=1$),  
\begin{equation*}
D_0 = P \cap Q = \left\{\left[\begin{array}{cccc}
a & x & y & z \\
0 & b & (bz-wx)a^{-1} & w \\
0 & 0 & a^{-1} & 0 \\ 
0 & 0 & -a^{-1}x b^{-1} & b^{-1}  
\end{array} \right] : a, b \in F^{\times}, x, y, z, w \in F\right\}
\end{equation*} 
and
\begin{equation*}
D_1 = \left\{\left[\begin{array}{cccc}
b & 0 & 0 & 0 \\
c & d & 0 & y \\
0 & 0 & b^{-1} & -c b^{-1} d^{-1} \\ 
0 & 0 & 0 & d^{-1}  
\end{array} \right] : b, d\in F^{\times}, c, y\in F  \right\}.
\end{equation*}
In the notation of Proposition \ref{PdoublecosetGeneral} (applied with $n=2$), the element $w_1 = \sigma_1.$ We note the following lemmas.

\begin{lemma}\label{Spsi-OrbitinKlingenDegA}
	
	$D_0 \backslash P/S_{\psi} = \left\{I_4, \tau_1\right\},$ where $\tau_1$ is as in \eqref{tau1-def}.
\end{lemma}

\begin{proof}

	We consider the action of $S_{\psi}$ on $D_0 \backslash P.$ Note that $B_1$ normalizes $N$ and $D_0 = B_1N.$ Also, $D_0 \backslash P/S_{\psi} = B_1 N \backslash MN/ M_{\psi} N.$ Consequently, $B_1 N \backslash MN/ M_{\psi} N$ can be identified with $ B_1\backslash M/M_{\psi} $ which further can be identified with $ B \backslash GL_{2}(F)/O_{2}(F,C).$ 
	The statement of the lemma now follows from Remark \ref{KlingenCtrivialcoset}.\qedhere \end{proof}

\begin{lemma}\label{SpsiOrbitNontrivialKlingenGegA}

	$D_1 \backslash P/S_{\psi} = \left\{I_4, \tau_1\right\},$ where $\tau_1$ is as in \eqref{tau1-def}.
\end{lemma}

\begin{proof}
	
	Consider the action of $S_{\psi}$ on $D_1 \backslash P$ and put $N^0=	\left\{\left[\begin{array}{cc}
	I_2 & \left[\begin{array}{cc}
	0&0 \\
	0 & x
	\end{array}\right] \\
	&I_2  
	\end{array} \right] : x \in F \right\}.$  Observe that $\overline{B}_1$ normalizes $N^{0}$ and we have $D_1 = \overline{B}_1 N^0.$ Thus, $D_1 \backslash P/S_{\psi}$ is equal to $ \overline{B}_1 N^0 \backslash MN/M_{\psi} N.$ Also, $ \overline{B}_1 N^0 \backslash MN/ M_{\psi} N $ can be identified with $ \overline{B}_1 \backslash M/ M_{\psi}.$ On the other hand, $ \overline{B}_1 \backslash M/ M_{\psi}$ is identified with $\overline{B} \backslash GL_{2}(F)/ O_{2}(F,C).$  The lemma now follows from Proposition \ref{O2C-acts-GmodBbar}. \end{proof}

\noindent The next proposition gives double coset representatives for $Q\backslash Sp_{4}(F)/S_{\psi}.$ 

\begin{proposition}\label{S-psi-orbit-Sp4/Q}
	$Q\backslash Sp_{4}(F)/S_{\psi}=\{I_4, {\tau_1}^{-1}, {\sigma_1}^{-1}, (\tau_1\sigma_1)^{-1}\},$ where $\sigma_1$ and $\tau_1$ are as in \eqref{sigma1-def} and \eqref{tau1-def} respectively.
\end{proposition}

\begin{proof}
	We consider here the action of $S_{\psi}$ on $Q \backslash Sp_{4}(F).$	Put  $X_0= \langle e_1 \rangle$ and $X_1= \langle e_1^{\vee} \rangle.$ We apply Proposition \ref{PdoublecosetGeneral} (with $n=2$ and $k=1$) to see that $\sigma_j(X_0)=X_j$ for $j\in \{0,1\}.$ By Lemma \ref{lemmafororbitsGeneral} (applied with $n=2$ and  $k=1$), the orbits for the action of $P$ on $Q \backslash Sp_{4}(F)$ are $Orb_P(X_0)$ and $Orb_P(\sigma_1\cdot X_0).$ To apply Lemma \ref{hijgj}, we consider the $S_{\psi}$-action on both of these $P$-orbits one-by-one.
	
	\textit{Case}(1): Note that $Orb_P(X_0)= D_0 \backslash P,$ as $D_0$ is the stabilizer in $P$ of $X_0$ (by Proposition \ref{StabilizerfororbitsGeneral} applied with $n=2$ and $k=1$). We consider the $S_{\psi}$ action on $D_0\backslash P.$ By Lemma \ref{Spsi-OrbitinKlingenDegA}, we have $D_0 \backslash P/S_{\psi} = \left\{I_4, \tau_1\right\}.$ Thus, $Orb_P(X_0)=Orb_{S_{\psi}}(X_0) \bigsqcup Orb_{S_{\psi}}(\tau_1\cdot X_0).$

	\textit{Case}(2): $D_1$ is the stabilizer in $P$ of $X_1$ by Proposition \ref{StabilizerfororbitsGeneral}. Therefore, $Orb_P(X_1)= D_1 \backslash P$ and by Lemma \ref{SpsiOrbitNontrivialKlingenGegA}, we have $D_1 \backslash P/ S_{\psi}=\{I_4,\tau_1\}.$ Thus,  $$Orb_P(X_1)=Orb_{S_{\psi}}(X_1) \bigsqcup Orb_{S_{\psi}}(\tau_1\cdot X_1)=Orb_{S_{\psi}}(\sigma_1 \cdot X_0) \bigsqcup Orb_{S_{\psi}}(\tau_1 \sigma_1\cdot X_0).$$ 
	We apply Lemma \ref{hijgj} with  $X= Q \backslash Sp_{4}(F), G=Sp_{4}(F),H=P, G_0=Q$ and $K=S_{\psi}$ to complete the proof of the proposition.
\end{proof}

\subsection{Some key subgroups arising in the computation}
\label{KeySubgroups}

In this section, we will define some key subgroups of $Sp_{4}(F)$ which arise repeatedly in the computation of the twisted Jacquet module and prove some properties concerning them. We begin by listing certain subgroups of $N.$

\subsubsection{Subgroups of $N$}\label{SubgroupsofN}

Put
\begin{equation*}	
N^0=	\left\{\left[\begin{array}{cc}
I_2 & \left[\begin{array}{cc}
0&0 \\
0 & x
\end{array}\right] \\
&I_2  
\end{array} \right] : x \in F \right\},	N^1 = \left\{ \left[\begin{array}{cc}
I_2 & \left[\begin{array}{cc}
0& y \\
y & z
\end{array} \right] \\
& I_2
\end{array}\right] : y, z \in F\right\}, \end{equation*}

\begin{equation*}
N^2=	\left\{\left[\begin{array}{cc}
I_2 & \left[\begin{array}{cc}
x&0 \\
0 & 0
\end{array}\right] \\
&I_2  
\end{array} \right] : x \in F \right\}, N^3:=N\cap U=\left\{\left[\begin{array}{cc}
I_2 & \left[\begin{array}{cc}
y & z \\
z & 0
\end{array} \right]\\
0 & I_2
\end{array}\right]:  y, z \in F \right\} \end{equation*} 
and
\begin{equation*}
N^4=\left\{\left[\begin{array}{cc}
I_2 &\left[\begin{array}{cc}
0& y \\
y &0
\end{array}\right]  \\
& I_2  
\end{array} \right] : y \in F\right\}. \end{equation*}

Note that the subgroups $N^0$ and $N^1$ have already featured in the proofs of Lemma \ref{SpsiOrbitNontrivialKlingenGegA} and Lemma \ref{Spsi-OrbitinSiegelDegA} respectively.

\subsubsection{Subgroups of $M$}\label{SubgroupsofM}

Put 
\begin{equation*}\label{DefnofM^0andM^1}
	M^1=\left\{\left[\begin{array}{cccc}
1 & & & \\
y & 1 & & \\
& & 1 & -y \\
& & & 1 \end{array} \right] : y\in F \right\}, M^2= \left\{\diag(a,d,a^{-1},d^{-1}) : d \in F^{\times}, a^2 =1\right\}
\end{equation*}
and
\begin{equation*}\label{DefnofM^2andM^3}
M^3= \left\{\diag(a,d,a^{-1},d^{-1}) :  a \in F^{\times}, d^2 =1\right\}.\end{equation*}

In the following lemma, we note the action of the character $\psi$ and the conjugation action of $M^2$ on each of the subgroups $N^j$ ($0\leq j \leq 4)$ of $N.$

\begin{lemma}\label{Psiaction-starcondition}
	\begin{enumerate}
		\item On the subgroups $N,$ $N^2$ and $N^3$ the character $\psi$ is non-trivial.
		
		\item On the subgroups $N^0$ and $N^1,$ $\psi$ is trivial.
		
		\item The group $M^2$ normalizes each of the groups $N^0,N^1, N^2$ and $N^4.$ Moreover, $M^2$ fixes $N^2$ point-wise under conjugation.
	\end{enumerate}	
\end{lemma}

\begin{proof}
	The proofs of (1) and (2) are trivial. To prove (3), let an element $m$ of $M^2$ be written as $m= \diag(a,d,a^{-1},d^{-1})$ where $a^2=1$ and $d\in F^{\times}.$ Note that if $n=\left[\begin{array}{cc}
	I_2 & \left[\begin{array}{cc}
	x & y \\
	y & z
	\end{array}\right] \\
	& I_2 \end{array}\right]$ we have
	$$ mnm^{-1}=\left[\begin{array}{cc}
	I_2 & \left[\begin{array}{cc}
	a^2x & ady \\
	ady & d^2z
	\end{array}\right] \\
	& I_2 \end{array}\right]. $$

\noindent The statement (3) follows from the following: \\
	\noindent (a) For  $n = \left[\begin{array}{cc}
	I_2 & \left[\begin{array}{cc}
	0 & 0 \\
	0 & x
	\end{array}\right] \\
	& I_2 \end{array}\right]\in N^0,$ we have $ mnm^{-1}= \left[\begin{array}{cc}
	I_2 & \left[\begin{array}{cc}
	0& 0 \\
	0 & d^2x
	\end{array} \right] \\
	& I_2
	\end{array}\right].$\\
	
	\noindent(b) For $n = \left[\begin{array}{cc}
	I_2 & \left[\begin{array}{cc}
	0 & y \\
	y & z
	\end{array}\right] \\
	& I_2 \end{array}\right]\in N^1,$ we have $mnm^{-1}= \left[\begin{array}{cc}
	I_2 & \left[\begin{array}{cc}
	0& ady \\
	ady & d^2z
	\end{array} \right] \\
	& I_2
	\end{array}\right].$
	
	\noindent(c) For $n = \left[\begin{array}{cc}
	I_2 & \left[\begin{array}{cc}
	x & 0 \\
	0 & 0
	\end{array}\right] \\
	& I_2 \end{array}\right]\in N^2,$ we have $mnm^{-1}= \left[\begin{array}{cc}
	I_2 & \left[\begin{array}{cc}
	a^2x& 0 \\
	0 & 0
	\end{array} \right] \\
	& I_2
	\end{array}\right].$ As $a^2=1,$  $mnm^{-1}=n$ and hence $M^2$ fixes $N^2.$

	\noindent (d) If $n = \left[\begin{array}{cc}
	I_2 & \left[\begin{array}{cc}
	0 & y \\
	y & 0
	\end{array}\right] \\
	& I_2 \end{array}\right]\in N^4,$ one has $mnm^{-1}= \left[\begin{array}{cc}
	I_2 & \left[\begin{array}{cc}
	0& ady \\
	ady & 0
	\end{array} \right] \\
	& I_2
	\end{array}\right].$
\end{proof}

\begin{proposition}\label{modular-character-computation}
	For an element $m=\diag(a,d,a^{-1},d^{-1})\in M^2$ and $m' = \left[\begin{array}{cc}
		\pm 1 & 0 \\
		y & d
		\end{array}\right] \in O_{2}(F,C)$ we have
	the following:
	
	\begin{enumerate}

		\item ${\rm mod}_{N^0}(m)=|d|_{_F}^2.$

		\item ${\rm mod}_{N^1}(m)=|d|_{_F}^3.$

		\item ${\rm mod}_{N^2}(m)=1.$

		\item ${\rm mod}_{N^4}(m)=|d|_{_F}.$
		
	 \item ${\rm mod}_{\overline{N_{1,1}}}(m')=|d|_{_F}.$

	\end{enumerate}	
\end{proposition}

\begin{proof}
	We will prove (2). To calculate ${\rm mod}_{N^1}(m),$ observe that $N^1$ is an abelian group which is isomorphic to the additive group $F^2.$ Hence, a Haar measure $d\mu(n)$ on $N^1$ is given by the product measure on $F^2$ induced from a Haar measure on $F.$ Note that if $n=\left[\begin{array}{cc}
	I_2 & \left[\begin{array}{cc}
	0 & y \\
	y & z
	\end{array}\right] \\
	& I_2 \end{array}\right] \in N^1$ we have $m^{-1}nm = \left[\begin{array}{cc}
	I_2& \left[\begin{array}{cc}
	0 & \frac{y}{ad}  \\
	\frac{y}{ad} & \frac{z}{d^2} 
	\end{array}\right] \\
	& I_2 
	\end{array}\right].$ 
	By a change of variables $y\mapsto ady$ and $z\mapsto d^2z$ it is clear that $$\displaystyle \int_{N^1}f(m^{-1}nm)d\mu(n)=|a|_{_F}  |d|_{_F}^{3}\int_{N^1}f(n)d\mu(n)$$ for any locally constant and compactly supported function $f$ on $N^1.$ Since $a=\pm 1$ for an element in $M^2,$ we have ${\rm mod}_{N^1}(m)=|d|_{_F}^{3}$ by \eqref{modular-character-defn}. The proofs of assertions (1), (3), (4) and (5) are similar.\end{proof}

We note the following standard facts about modular characters of the groups $P,$ $M_{\psi}$ and $Q$ in the following two remarks.

\begin{remark}\label{mod-char-MPsi}
	We observe that ${\rm mod}_N$ is thought of as character of $P,$ which is the normalizer of $N$ in $Sp_{4}(F).$  Moreover, ${\rm mod}_N$ is precisely the modular character $\delta_P$ of the group $P$ which is given by $\delta_P(mn)=|\det(g)|_{F}^{3}$ where $m= \diag(g, ~^tg^{-1}) \in M$ with $g\in GL_{2}(F)$ and $n \in N$ (see for instance \cite[Pg. 4]{Gustafson}). 
	
	\begin{enumerate}
		\item Put  $g=\left[\begin{array}{cc}a& 0 \\c& d\end{array}\right]$ with $a^2=1, d\in F^{\times}$ and $c\in F.$ For any $m=  \left[\begin{array}{cc}
		g & \\
		& ^tg^{-1}
		\end{array}\right]\in M_{\psi},$ it follows that ${\rm mod}_N(m)=|d|_{_F}^{3}.$
		Since $M_{\psi}=M^2\ltimes M^1,$ we can deduce that  $\delta_{M_{\psi}}(m')= {\rm mod}_{M^1}(m') =|d|_{_F}=\delta_{_{O_{2}(F,C)}}(g)$ for any $m'={\rm diag}(a,d,a^{-1},d^{-1}) \in M^2.$
			
			\item Write $m={\rm diag}(a,d,a^{-1},d^{-1}) \in M^2$ with $a^2=1$ and $d\in F^{\times}$. One has ${\rm mod}_N(m)=|d|_{_F}^{3}.$
		
	\end{enumerate} 

\end{remark}

\begin{remark}\label{mod-char-MPsi-Klingen-also}  The character ${\rm mod}_U$ is  the modular character $\delta_Q$ of $Q$ given by $\delta_Q(lu)=|t|_{_F}^{4}$ where $l=\left[\begin{array}{cccc}
	t &  &  &  \\
	& a &  & b \\
	&  & t^{-1} &  \\
	& c &  & d 
	\end{array}\right]  \in L$ with $t \in F^{\times},$ $\left[\begin{array}{cc}
	a & b \\
	c & d 
	\end{array}\right] \in SL_{2}(F)$ and $u\in U.$ Consequently, we have the following:
	
	\begin{enumerate}
		
		\item If $m={\rm diag}(a,d,a^{-1},d^{-1}) \in M^2$ where $a^2=1$ and $d\in F^{\times}.$ Then, ${\rm mod}_{U}(m)=1.$	
		
		\item If $m'={\rm diag}(a,d,a^{-1},d^{-1}) \in M^3$ where $d^2=1$ and $a\in F^{\times}.$ Then, ${\rm mod}_{U}(m')=|a|_{_F}^4.$
		
	\end{enumerate}
	
\end{remark}

\subsection{Decomposability Conditions}
\label{Decomposability}

In this paragraph, we verify the decomposability conditions (see \S \ref{geometriclemma}) which need to be verified to apply the Geometric Lemma. We shall write down explicitly all the conditions we need to verify for the case of the Siegel parabolic subgroup and the Klingen parabolic subgroup of $Sp_{4}(F).$ Recall that by $w(H),$ we mean the conjugate $wHw^{-1}$ of a subgroup $H.$

\subsubsection{The conditions for the Siegel parabolic}\label{Decomp-Siegel-conditions}
In the case of inducing from the Siegel parabolic subgroup $P,$ we have to show for each $w\in \{I_4,\sigma_1, \tau_1\sigma_1,\sigma_2\},$ the following holds:

\begin{enumerate}
	\item $w(P) \cap S_{\psi}=(w(P) \cap M_{\psi})\cdot (w(P) \cap N).$	
	
	\item $w(M) \cap S_{\psi}=(w(M) \cap M_{\psi})\cdot (w(M) \cap N).$	
	
	\item $w(N) \cap S_{\psi}=(w(N) \cap M_{\psi})\cdot (w(N) \cap N).$	
	
	\item $w^{-1}(S_{\psi})\cap P=(w^{-1}(S_{\psi}) \cap M)\cdot (w^{-1}(S_{\psi}) \cap N).$
	
	\item $w^{-1}(M_{\psi})\cap P=(w^{-1}(M_{\psi}) \cap M)\cdot (w^{-1}(M_{\psi}) \cap N).$
	
	\item $w^{-1}(N)\cap P=(w^{-1}(N) \cap M)\cdot (w^{-1}(N) \cap N).$
\end{enumerate}

\subsubsection{The conditions for the Klingen parabolic}\label{Decomp-Klingen-conditions}
In the case of the Klingen parabolic subgroup $Q,$ we have to prove that for each $w\in \{I_4, \tau_1,\sigma_1, \tau_1\sigma_1\},$ the following holds:

\begin{enumerate}
	\item $w(Q) \cap S_{\psi}=(w(Q) \cap M_{\psi})\cdot (w(Q) \cap N).$
	
	\item $w(L) \cap S_{\psi}=(w(L) \cap M_{\psi})\cdot (w(L) \cap N).$	
	
	\item $w(U) \cap S_{\psi}=(w(U) \cap M_{\psi})\cdot (w(U) \cap N).$	
	
	\item $w^{-1}(S_{\psi})\cap Q=(w^{-1}(S_{\psi}) \cap L)\cdot (w^{-1}(S_{\psi}) \cap U).$
	
	\item $w^{-1}(M_{\psi})\cap Q=(w^{-1}(M_{\psi}) \cap L)\cdot (w^{-1}(M_{\psi}) \cap U).$
	
	\item $w^{-1}(N)\cap Q=(w^{-1}(N) \cap L)\cdot (w^{-1}(N) \cap U).$
\end{enumerate}

\begin{remark}\label{Oneside-decomp}
	In all the above cases, there is one generality. Suppose $H_1,H_2$ are two subgroups of a group $H$ with $H_2$ normal in $H$ and  $H=H_1H_2.$ Then given any subgroup $J$, it is easy to see that $H_2\cap J$ is normal in $H\cap J$ and that $(H_1\cap J) \cdot (H_2\cap J)$ is a subgroup of $H\cap J.$ In view of this, it is sufficient to check that the group on the left hand side of the equalities is a subgroup of the product of the groups on the respective right hand side.
\end{remark}

\begin{lemma}\label{Decomposability-verification-Siegel} For $w\in \{I_4,\sigma_1, \tau_1\sigma_1,\sigma_2\},$ the following statements hold:
	\begin{enumerate}
		\item[(a)] The subgroups $w(P),w(M)$ and $w(N)$ are decomposable with respect to the pair of subgroups $(M_{\psi},N).$
		
		\item[(b)] The subgroups $w^{-1}(S_{\psi}),w^{-1}(M_{\psi})$ and $w^{-1}(N)$ are decomposable with respect to the pair of subgroups $(M,N).$
	\end{enumerate}
\end{lemma}

\begin{proof}
	In view of Remark \ref{Oneside-decomp}, it is sufficient to check that for each $w,$ the left hand side of each equality in \S \ref{Decomp-Siegel-conditions} is contained in the right hand side. 
	
	\textit{Case}(1): Fix $w=I_4.$ Both (a) and (b) are easily verified here as conditions (1)--(6) in \S \ref{Decomp-Siegel-conditions} are trivial to check.
	
	\textit{Case}(2): Fix $w=\tau_1\sigma_1.$ 	
	To prove (a), we need to verify  the conditions (1)--(3) in 	\S \ref{Decomp-Siegel-conditions} of which we first verify $(1).$ It is sufficient to check that for an element $p\in P,$ if $wpw^{-1}\in S_{\psi}$ then it can be written as a product $mn$ where $m\in wPw^{-1}\cap M_{\psi}$ and $n\in wPw^{-1}\cap N.$ To this end, we write $p=\left[\begin{array}{cc}
	g & X\\
	& ^tg^{-1} 
	\end{array}\right]$ with $g=\left[\begin{array}{cc}
	x & y \\
	z & v
	\end{array}\right]\in GL_{2}(F)$ and $X=\left[\begin{array}{cc}
	r & s\\
	t& u 
	\end{array}\right]\in M_2(F).$ Then, $wpw^{-1}=\left[\begin{array}{cccc}
	v & -t & u & z \\
	0 & \frac{v}{xv-yz} & \frac{z}{xv-yz} & 0 \\
	0 & \frac{y}{xv-yz} &\frac{x}{xv-yz} & 0 \\ 
	y & -r & s & x  
	\end{array} \right].$ Assume that $wpw^{-1}\in S_{\psi}.$ This forces $y=0$ and $r=0.$ This further implies that $v^2 =1$ and $t=0=s.$ Then, we can write $ wpw^{-1}=\left[\begin{array}{cc}
	A & Y \\
	& B 
	\end{array}\right]$ where $A=  \left[\begin{array}{cc}
	v & 0 \\
	0 & x^{-1}
	\end{array}\right], B=  \left[\begin{array}{cc}
	v^{-1} & 0 \\
	0 & x 
	\end{array}\right]$ and $Y= \left[\begin{array}{cc}
	u & z \\
	\frac{z}{xv} & 0 \end{array}\right].$ It is trivial to check that $A^{-1}Y$ is symmetric and $B=~^tA^{-1}.$
	Thus we can write, $$wpw^{-1}=\left[\begin{array}{cc}
	A & Y\\
	& ^tA^{-1} 
	\end{array}\right]=\left[\begin{array}{cc}
	A &  \\
	& ^tA^{-1} 
	\end{array}\right] \left[\begin{array}{cc}
	I_2 & A^{-1}Y\\
	& I_2 
	\end{array}\right].$$ Our argument would be complete, if we show that $\left[\begin{array}{cc}
	A &  \\
	& ^tA^{-1} 
	\end{array}\right]\in w(P)\cap M_{\psi}$ and $\left[\begin{array}{cc}
	I_2 & A^{-1}Y\\
	& I_2 
	\end{array}\right]\in w(P)\cap N.$ Choose $p_1=\left[\begin{array}{cc}
	g_1 &  \\
	&  ~^tg_1^{-1}
	\end{array}\right]\in P$ where $g_1=\left[\begin{array}{cc}
	x & 0  \\
	0 & v
	\end{array}\right]$ with $v^2=1.$ It is easy to see that $wp_1w^{-1}=\left[\begin{array}{cc}
	A &  \\
	& ^tA^{-1} 
	\end{array}\right].$ The condition $wpw^{-1}\in S_{\psi}$ implies that $wp_1w^{-1}\in M_{\psi}.$ Similarly, 
	choose $p_2=\left[\begin{array}{cc}
	g_2 & Y_1 \\
	0 & ~^tg_2^{-1}
	\end{array}\right] \in P$ where $g_2=\left[\begin{array}{cc}
	1 & 0 \\
	vz & 1
	\end{array}\right]$ and $Y_1=\left[\begin{array}{cc}
	0 & 0 \\
	0 & vu
	\end{array}\right]$ with $v^2=1.$ Again, it is easy to see that $wp_2w^{-1}=\left[\begin{array}{cc}
	I_2 & A^{-1}Y \\
	0 & I_2
	\end{array}\right]\in N$ since $wpw^{-1}\in S_{\psi}.$ The conditions (2) and (3) are checked similarly and the proof is omitted.

	To prove (b), we need to verify the conditions (4)--(6) in \S \ref{Decomp-Siegel-conditions}. We verify (4) first. To do so, it is sufficient to check that for an element $h\in S_{\psi},$ if $w^{-1}hw\in P$ then it can be written as a product $mn$ where $m\in w^{-1}S_{\psi} w\cap M$ and $n\in w^{-1}S_{\psi} w\cap N.$ Write $h=\left[\begin{array}{cc}
	g & X\\
	& ^tg^{-1} 
	\end{array}\right]$ with $g=\left[\begin{array}{cc}
	x & y \\
	z & v
	\end{array}\right]\in GL_{2}(F)$ and $X=\left[\begin{array}{cc}
	r & s\\
	t& u 
	\end{array}\right]\in M_2(F).$ As $h\in S_\psi,$ we have $x^2=1$ and $y=0.$ We get $w^{-1}hw=\left[\begin{array}{cccc}
	v^{-1} & 0 & 0 & 0 \\
	s & x & 0 & r \\
	-u & -z &v & -t \\ 
	-zx^{-1}v^{-1} & 0 & 0 & x^{-1}  
	\end{array} \right].$
	Assuming $w^{-1}hw\in P$  gives $u=0=z$ and $s=txv^{-1}.$
	We can then write $w^{-1}hw=\left[\begin{array}{cc}
	A & Y \\
	& B 
	\end{array}\right]$ where $A=  \left[\begin{array}{cc}
	v^{-1} & 0 \\
	txv^{-1} & x
	\end{array}\right], Y= \left[\begin{array}{cc}
	0& 0 \\
	0 & r \end{array}\right]$ and  $B=  \left[\begin{array}{cc}
	v & -t \\
	0 & x^{-1} 
	\end{array}\right].$ Consequently, $A^{-1}Y$ is symmetric and $B=~^tA^{-1}.$ So, $$w^{-1}hw=\left[\begin{array}{cc}
	A & Y\\
	& ^tA^{-1} 
	\end{array}\right]=\left[\begin{array}{cc}
	A &  \\
	& ^tA^{-1} 
	\end{array}\right] \left[\begin{array}{cc}
	I_2 & A^{-1}Y\\
	& I_2 
	\end{array}\right].$$ We will show that $\left[\begin{array}{cc}
	A &  \\
	& ^tA^{-1} 
	\end{array}\right]\in w^{-1}(S_{\psi})\cap M$ and $\left[\begin{array}{cc}
	I_2 & A^{-1}Y\\
	& I_2 
	\end{array}\right]\in w^{-1}(S_{\psi})\cap N.$ Choose $s_1=\left[\begin{array}{cc}
	h_1 &  Y_2 \\
	&  ~^th_1^{-1}
	\end{array}\right]\in S_\psi$ where $h_1=\left[\begin{array}{cc}
	x & 0  \\
	0 & v
	\end{array}\right]$ and $Y_2=\left[\begin{array}{cc}
	0 & v^{-1}xt \\
	t & 0
	\end{array}\right].$ Since $x^2=1,$ we obtain  $w^{-1}s_1w=\left[\begin{array}{cc}
	A &  \\
	& ^tA^{-1} 
	\end{array}\right].$ 
	Next, choose $s_2=\left[\begin{array}{cc}
	I_2 & Y_3 \\
	0 & I_2
	\end{array}\right]\in S_\psi$ where $Y_3=\left[\begin{array}{cc}
	rx^{-1} & 0 \\
	0 & 0
	\end{array}\right]$ with $x^2=1.$ We obtain $w^{-1}s_2w=\left[\begin{array}{cc}
	I_2 & A^{-1}Y \\
	0 & I_2
	\end{array}\right].$ This verifies $(4).$ The proofs of conditions (5) and (6) are similar and omitted.

	For the remaining $w,$ i.e., $w\in \{ \sigma_1, \sigma_2\}$ the verification is similar. We shall describe the groups which appear in the decomposability conditions (1) and (4) in both these cases.

	\textit{Case}(3): Fix $w=\sigma_1.$
	We describe each of the groups which appear in \S \ref{Decomp-Siegel-conditions} (1). 
	\begin{enumerate}
		\item[(i)] $w(P) \cap S_{\psi} = \left\{\left[\begin{array}{cccc}
		x^{-1} & 0 &  & zx^{-1} v^{-1} \\
		t & v & z  & u \\
		&  & x & -t x v^{-1} \\ 
		&  & 0 & v^{-1}  
		\end{array} \right] : v\in F^{\times}, t,u,z\in F, x^2=1 \right\}.$
		
		\item[(ii)] $w(P) \cap M_{\psi} = \left\{\left[\begin{array}{cccc}
		x^{-1} & 0 &  & \\
		t & v &   &  \\
		&  & x & -t x v^{-1} \\ 
		&  & 0 & v^{-1}  
		\end{array} \right] : v\in F^{\times}, t\in F, x^2=1 \right\}.$
		
		\item[(iii)] $w(P) \cap N = N^1.$
		
	\end{enumerate}	
	It is easy to check that $w(P) \cap S_{\psi}=(w(P) \cap M_{\psi})\cdot (w(P) \cap N).$	We next describe the groups which correspond to \S \ref{Decomp-Siegel-conditions} (4).
	\begin{enumerate}

		\item[(iv)] $w^{-1}(S_{\psi}) \cap P =w(P)\cap S_{\psi}.$
		
		\item[(v)] $w^{-1}(S_{\psi}) \cap M = w(P) \cap M_{\psi}.$
		
		\item[(vi)] $w^{-1}(S_{\psi}) \cap N = N^1.$
	\end{enumerate}
	
	The remaining conditions (2), (3), (5) and (6) in \S \ref{Decomp-Siegel-conditions} are in fact easier and their proofs are omitted.

	\textit{Case}(4): Fix $w=\sigma_2.$ Here, we note that $w(P)\cap S_{\psi}=w(P)\cap M_{\psi}$ and $w(P)\cap N=\{I_4\}.$ Similarly, $w^{-1}(S_{\psi})\cap P=w^{-1}(S_{\psi})\cap M$ and $w^{-1}(S_{\psi})\cap N=\{I_4\}.$ The remaining conditions are trivial and omitted.\qedhere
	
\end{proof}

In the next lemma, we verify the decomposability conditions required for the Klingen parabolic case. 

\begin{lemma}\label{Decomposability-verification-Klingen} For $w\in \{I_4, \sigma_1, \tau_1, \tau_1 \sigma_1\},$ we have the following:
	
	\begin{enumerate}
		\item[(a)] The subgroups $w(Q),w(L)$ and $w(U)$ are decomposable with respect to the pair of subgroups $(M_{\psi},N).$
		
		\item[(b)] The subgroups $w^{-1}(S_{\psi}),w^{-1}(M_{\psi})$ and $w^{-1}(N)$ are decomposable with respect to the pair of subgroups $(L,U).$
	\end{enumerate}
\end{lemma}

\begin{proof}
	We invoke Remark \ref{Oneside-decomp} by which we need only show that for each $w,$ the left hand side of each equality in \S \ref{Decomp-Klingen-conditions} is contained in the right hand side. 
	
	\textit{Case}(1): Fix $w=\tau_1\sigma_1.$ We first prove (a), for which we need to verify \S \ref{Decomp-Klingen-conditions} (1)--(3).  To prove (1), take $q=\left[\begin{array}{cccc}
	t & x & y & z \\
	& a & v & b \\
	&  & t^{-1} &  \\
	&  c& u & d 
	\end{array}\right]\in Q.$ Then, 
	$wqw^{-1}= \left[\begin{array}{cccc}
	a & -v & b & 0 \\
	0& t^{-1} & 0 & 0 \\
	c&  -u& d &0  \\
	x&  -y& z & t 
	\end{array}\right].$ Assume that $wqw^{-1}\in S_{\psi}.$ This gives $x=y=0=u=c$ and in turn ${a}^2=1$ and $v=0.$ Moreover, $q\in Q$  yields that $d = {a}^{-1}$ and $z = tv {a}^{-1}=0.$ With these, we can write $ wqw^{-1}=\left[\begin{array}{cc}
	A & Y \\
	& B 
	\end{array}\right]$ where $A=  \left[\begin{array}{cc}
	a & 0 \\
	0 & t^{-1}\end{array}\right], B=  \left[\begin{array}{cc} a^{-1}  & 0 \\ 0 &t \end{array}\right]$ and $Y= \left[\begin{array}{cc}
	b &0\\
	0 & 0
	\end{array}\right].$  It easy to see that $B=~^tA^{-1}$ and $A^{-1}Y$ is symmetric. So, $$wqw^{-1}=\left[\begin{array}{cc}
	A & Y\\
	& ^tA^{-1} 
	\end{array}\right]=\left[\begin{array}{cc}
	A &  \\
	& ^tA^{-1} 
	\end{array}\right] \left[\begin{array}{cc}
	I_2 & A^{-1}Y\\
	& I_2 
	\end{array}\right].$$ Next, we show that $\left[\begin{array}{cc}
	A &  \\
	& ^tA^{-1} 
	\end{array}\right]\in w(Q) \cap M_{\psi}$ and $\left[\begin{array}{cc}
	I_2 & A^{-1}Y\\
	& I_2 
	\end{array}\right]\in w(Q)\cap N.$ Choose $q_1=\diag(t,a,t^{-1},a^{-1}) \in Q$ with ${a}^2=1.$ We see that $wq_1w^{-1}=\left[\begin{array}{cc}
	A &  \\
	& ^tA^{-1} 
	\end{array}\right]$ and $wqw^{-1}\in S_{\psi}$ yields $wq_1w^{-1}\in M_{\psi}.$ Choose $q_2=\left[\begin{array}{cc}
	I_2 & Y_2 \\
	0 & I_2
	\end{array}\right]\in Q$ where $Y_2=\left[\begin{array}{cc}
	0 & 0\\
	0 &  a^{-1}b 
	\end{array}\right]$ with ${a}^2=1.$ We get $wq_2w^{-1}=\left[\begin{array}{cc}
	I_2 & A^{-1}Y \\
	0 & I_2
	\end{array}\right]\in N$ since $wqw^{-1}\in S_{\psi}.$ We have thus verified \S \ref{Decomp-Klingen-conditions} (1). The proofs of (2) and (3) are easier and omitted.

	To prove (b), we need to verify \S \ref{Decomp-Klingen-conditions} (4)--(6). To prove (4), it is enough to verify that for an element $h \in S_{\psi},$ if $w^{-1}hw \in Q$ then it can be written as a product $mn$ where $m\in w^{-1} S_{\psi} w \cap L$ and $n\in w^{-1} S_{\psi} w \cap U.$ Write $h=\left[\begin{array}{cc}
	g & X\\
	& ^tg^{-1} 
	\end{array}\right]$ with $g=\left[\begin{array}{cc}
	x & y \\
	z & v
	\end{array}\right]\in GL_{2}(F)$ and $X=\left[\begin{array}{cc}
	r & s\\
	t& u 
	\end{array}\right]\in M_2(F).$  As $h\in S_\psi,$ we have $x^2=1$ and $y=0.$
	We note that
	$w^{-1}hw=\left[\begin{array}{cccc}
	v^{-1} & 0 & 0 & 0 \\
	s& x & 0 & r \\
	-u&  -z& v & -t\\
	-zx^{-1}v^{-1}&  0& 0 & x^{-1} 
	\end{array}\right].$ It is easy to see that $w^{-1} S_{\psi} w \cap U = \{I_4\}.$ 
	Suppose $w^{-1}hw \in Q.$ This implies that $s=0=u=z=t.$ Hence, $ w^{-1}hw =\left[\begin{array}{cccc}
	v^{-1}&  &0 &0 \\
	& x & 0 & r \\
	& & v &   \\
	& & & x^{-1}
	\end{array}\right]$ and it belongs to $L.$ By taking $s_1=\left[\begin{array}{cc}
	h_1 &  Y\\
	& ~^t{h_1}^{-1}  
	\end{array}\right]\in S_{\psi}$ where $h_1 = \left[\begin{array}{cc}
	x& 0\\
	0& v
	\end{array}\right]$ with $x^2=1$ and $Y=\left[\begin{array}{cc}
	r& 0\\
	0& 0
	\end{array}\right],$ we get $w^{-1}s_1w= \left[\begin{array}{cccc}
	v^{-1}&  &0 &0 \\
	& x & 0 & r \\
	& & v &   \\
	& & & x^{-1}
\end{array}\right].$ 
	This proves (4). The verifications for (5) and (6) are similar and omitted.
	
	\textit{Case}(2): Fix $w = I_4.$ We describe the groups which appear in \S \ref{Decomp-Klingen-conditions} (1) and (4).  The groups appearing in (1) are as follows:
	\begin{enumerate}
		\item[(i)] $Q \cap S_{\psi} = \left\{\left[\begin{array}{cccc}
		t & 0 & y & z \\
		0 & a & azt^{-1}  & b \\
		&  & t^{-1} & 0 \\ 
		&  & 0 & a^{-1}  
		\end{array} \right] : a\in F^{\times}, b,y,z\in F, t^2=1 \right\}.$
		
		\item[(ii)] $Q \cap M_{\psi} = \left\{\diag(t,a,t^{-1},a^{-1}) : a\in F^{\times}, t^2=1 \right\}.$
		
		\item [(iii)]$Q\cap N = N.$
	\end{enumerate}
	It is easy to  see that $Q\cap S_{\psi}= (Q\cap M_{\psi}) \cdot (Q \cap N).$ The groups corresponding to (4) are as follows:
	\begin{enumerate}		
		\item[(iv)] $S_{\psi} \cap Q = Q \cap S_{\psi}.$
		
		\item[(v)] $S_{\psi} \cap L = \left\{\left[\begin{array}{cccc}
		t & 0 & 0 &0 \\
		0 & a & 0  &b  \\
		&  & t^{-1} & 0 \\ 
		&  & 0 & a^{-1}  
		\end{array} \right] : a\in F^{\times}, b \in F, t^2=1 \right\}.$
		
		\item[(vi)] $S_{\psi}\cap U = N^3.$
	\end{enumerate}
	It is easy to see that \S \ref{Decomp-Klingen-conditions} (4) holds as well. In a similar way the equalities (2), (3), (5) and (6) in \S \ref{Decomp-Klingen-conditions} can be proved and proofs are omitted.

	\textit{Case}(3): For $w = \sigma_1$ we have the following groups which appear in \S \ref{Decomp-Klingen-conditions} (1).
	\begin{enumerate}
		\item[(i)] $w(Q) \cap S_{\psi} = \left\{\left[\begin{array}{cccc}
		t^{-1} & 0 & 0 &0 \\
		-azt^{-1} & a & 0  &  b\\
		&  & t & z \\ 
		&  & 0 & a^{-1}  
		\end{array} \right] : a\in F^{\times}, z,b \in F, t^2=1 \right\}.$
		
		\item[(ii)] $w(Q) \cap M_{\psi} = \left\{\left[\begin{array}{cccc}
		t^{-1} & 0 &  & \\
		-azt^{-1} & a &   &  \\
		&  & t & z \\ 
		&  & 0 & a^{-1}  
		\end{array} \right] : a\in F^{\times}, z \in F, t^2=1 \right\}.$
		
		\item[(iii)] $w(Q) \cap N = N^0.$
	\end{enumerate}
	The groups which appear in \S \ref{Decomp-Klingen-conditions} (4) are as follows:
	\begin{enumerate}
		\item[(iv)] $w^{-1}(S_{\psi}) \cap Q = \left\{\left[\begin{array}{cccc}
		t^{-1} & 0 & 0& zt a^{-1} \\
		0 & a & z  & u \\
		&  & t & 0 \\ 
		&  & 0 & a^{-1}  
		\end{array} \right] : a\in F^{\times}, u, z\in F, t^2=1 \right\}.$
		
		\item[(v)] $w^{-1}(S_{\psi}) \cap L = \left\{\left[\begin{array}{cccc}
		t^{-1} & 0 & 0 &0 \\
		0 & a & 0  &u  \\
		&  & t & 0 \\ 
		&  & 0 & a^{-1}  
		\end{array} \right] : a\in F^{\times}, u \in F, t^2=1 \right\}.$
		
		\item[(vi)] $w^{-1}(S_{\psi}) \cap U = N^4.$
	\end{enumerate}
		In these cases also, the decomposability conditions are easily verified. The verification of (2), (3), (5) and (6) in \S \ref{Decomp-Klingen-conditions} can be done similarly and are omitted.

	\textit{Case}(4): Fix $w=\tau_1.$ We have the following groups relevant to \S \ref{Decomp-Klingen-conditions} (1).
	\begin{enumerate}
		\item[(i)] $w(Q) \cap S_{\psi} = \left\{\left[\begin{array}{cccc}
		a & 0 & b & at^{-1}(btu+z) \\
		-atu & t & z  & y \\
		&  & a^{-1} & u \\ 
		&  & 0 & t^{-1}  
		\end{array} \right] : t\in F^{\times}, b,u,y,z\in F, a^2=1 \right\}.$
		
		\item[(ii)] $w(Q) \cap M_{\psi} = \left\{\left[\begin{array}{cccc}
		a & 0 &  & \\
		-atu & t &   &  \\
		&  & a^{-1} & u \\ 
		&  & 0 & t^{-1}  
		\end{array} \right] : t\in F^{\times},u \in F, a^2=1 \right\}.$
		
		\item[(iii)] $w(Q) \cap N = N.$
	\end{enumerate}
	Finally, the groups relevant to \S \ref{Decomp-Klingen-conditions} (4) are the following.
		\begin{enumerate}
		
		\item[(iv)] $w^{-1}(S_{\psi}) \cap Q = \left\{\left[\begin{array}{cccc}
		v & z & u & t \\
		0 & x & v^{-1}(-zr+xt)  & r\\
		&  & v^{-1} & 0 \\ 
		&  & -z x^{-1} v^{-1} & x^{-1} 
		\end{array} \right] : v\in F^{\times}, u,t,r,z\in F, x^2=1 \right\}.$
		
		\item[(v)] $w^{-1}(S_{\psi}) \cap L = \left\{\left[\begin{array}{cccc}
		v & 0 & 0 &0 \\
		0 & x & 0  &r \\
		&  & v^{-1} & 0 \\ 
		&  & 0 & x^{-1}  
		\end{array} \right] : v\in F^{\times}, r \in F, x^2=1 \right\}.$
		
		\item[(vi)] $w^{-1}(S_{\psi}) \cap U = U.$
	\end{enumerate}
	In this case also, \S \ref{Decomp-Klingen-conditions} (1) and (4) are verified easily. The equalities (2), (3), (5) and (6) in \S \ref{Decomp-Klingen-conditions} can be proved similarly and proofs are omitted.
\end{proof}

\section{Structure of the Twisted Jacquet modules}\label{DegenerateCharacter}

In this section, we compute the twisted Jacquet module of a principal series representation $\pi$ of $Sp_{4}(F)$ induced from the Siegel parabolic subgroup $P$ (resp. Klingen parabolic subgroup $Q$) with respect to the character $\psi_{_{C}}$ of $N$ defined by \eqref{CharacterCorrspondngtoC}.  We shall denote $\psi_{_C}$ by $\psi.$ Note that the twisted Jacquet module of a representation of $Sp_{4}(F)$ is a representation of the group $M_{\psi}$ which is isomorphic to $O_2(F,C).$ To make the structure concise, the final answer will be given as a representation of the group $O_2(F,C).$ To achieve this, we shall use certain identifications which we describe in \S \ref{relevantembeddings} below.

\subsection{Some isomorphisms}\label{relevantembeddings}

Let us denote the maps given by \eqref{embedingofleviSp_{2n}(F)} and  \eqref{embedingofleviSp_{2n}(F)-Siegel} by $\alpha$ and $\beta$ respectively i.e., 
\begin{equation}\label{alpha}
\alpha\left(\left(t, \left[\begin{array}{cc}
a & b \\
c & d
\end{array}\right] \right)\right)=\left[\begin{array}{cccc}
t & & & \\
& a & & b \\
& & t^{-1} & \\
& c & & d
\end{array}\right]
\end{equation}
for all $t\in F^{\times}$ and 
$\left[\begin{array}{cc}
a & b \\
c & d
\end{array}\right] \in SL_{2}(F)$ and 

\begin{equation}\label{beta}\beta(g)=\left[\begin{array}{cc}
g &  \\
& ~^tg^{-1}  \\
\end{array}\right]\end{equation} for all $g\in GL_{2}(F).$
We identify certain subgroups of $GL_{2}(F)$ and $F^{\times}\times SL_{2}(F)$ with subgroups of $Sp_{4}(F).$  

\begin{enumerate}
	
	\item We have $\alpha(\{1\}\times N_{1,1})=N^0.$ With this identification, the group $N_{1,1}$ can be thought of as a subgroup of $Sp_{4}(F)$ when identified with $N^0$ via $\alpha.$
	
	\item Let $\overline{N_{1, 1}}$ and $T_{2}(F,C)$ be as in \eqref{N_1,1andbardefn} and \eqref{DefofT2FC} respectively. Then,
	\begin{enumerate}
		\item[(a)] $\beta(O_{2}(F,C))=M_{\psi}.$ 
		
		\item[(b)] $\beta(T_{2}(F,C))=M^{2}.$ 
		
		\item[(c)] $\beta(\overline{N_{1, 1}})=M^{1}.$
	\end{enumerate} 
	
	\item Recall that $O_{2}(F, C)$ is the semi-direct product of $T_{2}(F, C)$ and $\overline{N_{1, 1}}.$ This yields $M_{\psi} = M^2 \ltimes M^1.$ 
	
\end{enumerate}

For an element $y\in F,$ write $
n_y=\left[\begin{array}{cc}
1 &  \\
y & 1
\end{array}\right]$ and $n^y=\left[\begin{array}{cc}
1 & y \\
& 1
\end{array}\right].$ Then the group $M^{1}=\{\diag(n_y,n^{-y}):y\in F\}.$ We make one final remark before we end this paragraph.

\begin{remark}\label{quotientO_2}
	By \S \ref{relevantembeddings} (3), any representation of $T_2(F,C)$ can be regarded as a representation of $O_2(F,C)$ by composing with the quotient map from $O_2(F,C)$ to $T_2(F,C)$ and the resulting representation is trivial on $\overline{N_{1,1}}.$
\end{remark}

\subsection{The characters appearing in the computation}\label{modularcharacters} 

In this paragraph, we shall recall the formula for the characters appearing in \S \ref{geometriclemma} which needs to be computed for the case of both the Siegel and the Klingen parabolic subgroups of $Sp_{4}(F).$

\begin{enumerate}
	\item To each $w\in \{I_4,\sigma_1, \tau_1\sigma_1,\sigma_2\}$ for which \eqref{starcondition} holds, we put
	\begin{enumerate}
		\item[(i)] $\varepsilon_1(w) = {\rm mod}_{N}^{\frac{1}{2}} \cdot {\rm mod}_{w^{-1}S_{\psi}w \cap N}^{-\frac{1}{2}},$
		
		\item[(ii)] $\varepsilon_2(w) = {\rm mod}_{N}^{\frac{1}{2}} \cdot {\rm mod}_{w P w^{-1} \cap N}^{-\frac{1}{2}},$ and
		
		\item[(iii)] $\varepsilon(w)= \varepsilon_1(w) \cdot w^{-1}(\varepsilon_2(w))$ 
	\end{enumerate}
	which are characters of the groups ${w}^{-1}M_{\psi} w \cap M, wM{w}^{-1} \cap M_{\psi}$ and ${w}^{-1}M_{\psi} w \cap M$ respectively.

	\item Similarly, corresponding to each $w\in \{I_4, \tau_1,\sigma_1, \tau_1\sigma_1\}$ for which \eqref{starcondition} holds, we put
	\begin{enumerate}
		\item[(i)] $\varepsilon_1(w) = {\rm mod}_{U}^{\frac{1}{2}} \cdot {\rm mod}_{w^{-1}S_{\psi}w \cap U}^{-\frac{1}{2}},$ 
		
		\item[(ii)] $\varepsilon_2(w) = {\rm mod}_{N}^{\frac{1}{2}} \cdot {\rm mod}_{w Q w^{-1} \cap N}^{-\frac{1}{2}},$ and
		
		\item[(iii)] $\varepsilon(w)= \varepsilon_1(w) \cdot w^{-1}(\varepsilon_2(w))$ 
	\end{enumerate}
	which are characters of the groups ${w}^{-1}M_\psi w \cap L, wL{w}^{-1} \cap M_\psi$ and  ${w}^{-1}M_\psi w \cap L$ respectively.
\end{enumerate}

\subsection{Structure in the Siegel parabolic case} 

We are now ready to state and prove the first main theorem of our article, which gives the structure of the twisted Jacquet module of a principal series representation induced from the Siegel parabolic.

\begin{theorem}\label{SiegelcaseforDegA-I}
	Fix a non-trivial character $\psi_{_0}$ of the additive group $F.$ Let $P=MN$ 
	denote the Siegel parabolic subgroup of the symplectic group $Sp_{4}(F)$ and 
	let $\rho\in \Irr(GL_{2}(F))$ be regarded as a representation of the Levi 
	subgroup $M$ of $P.$ Denote $\pi=\ind_P^{Sp_{4}(F)}(\rho).$ Let $\psi$ be the character of $N$ defined by \eqref{CharacterCorrspondngtoC}. The following statements hold:
	
	\begin{enumerate}
		\item If $\rho$ is not cuspidal,  $r_{N,\psi}(\pi)$
		sits in the following short exact sequence of $O_{2}(F,C)$-modules:
		\begin{equation}\label{structure-Siegel}
		0\to  \widetilde{\rho}_{|_{O_{2}(F,C)}} \to r_{N,\psi}(\pi)\to  \rho_{_0} \to 0,
		\end{equation}
		where $\rho_{_0}$ is the representation of $O_{2}(F,C)$ obtained from $(\delta_{_{O_{2}(F,C)}}^{\frac{1}{2}} \cdot r_{\overline{N_{1, 1}}}(\rho))_{|_{T_{2}(F,C)}}$ via composing with the quotient map.

		\item If $\rho$ is cuspidal, $r_{N,\psi}(\pi) =  \widetilde{\rho}_{|_{O_{2}(F,C)}}$ as $O_{2}(F,C)$-modules. 
	\end{enumerate}
\end{theorem}

\begin{proof}
	We prove the theorem in two steps. First, we apply the geometric lemma to obtain all the possible components of $r_{N,\psi}(\pi).$ Then, we divide the proof into two separate cases as to whether $\rho$ is cuspidal or not. 
	
	\noindent \textbf{Step} (1). By Lemma \ref{Decomposability-verification-Siegel}, Proposition \ref{S-psi-orbit-Sp4/P} and \S \ref{geometriclemma} we may apply the Geometric Lemma to $\pi.$ In this case, $\mathcal{W} = \{I_4, \sigma_1, \tau_1\sigma_1, \sigma_2\}$ by Proposition \ref{S-psi-orbit-Sp4/P}. We need to compute $\Phi_w$ for each $w \in \mathcal{W}.$  
	We begin the proof by showing that \eqref{starcondition} does not hold for $I_4$ and $\tau_1\sigma_1.$  For $w = I_4,$ the group $wNw^{-1} \cap N$ is $N.$ On the other hand, for $w = \tau_1 \sigma_1,$ we have $ wNw^{-1} \cap N =N^2 .$ 
		By Lemma \ref{Psiaction-starcondition} (1) the character $\psi$ of $N$ is non-trivial on both the subgroups $N$ and $N^2.$ Hence, neither $I_4$ nor $\tau_1\sigma_1$ contribute to the twisted Jacquet module. Thus, $\Phi_{I_4}=0$ and $\Phi_{\tau_1\sigma_1}=0.$
	
	For the remaining two $w$'s i.e., for $w=\sigma_1$ and $w=\sigma_2,$ we have $wNw^{-1}\cap N$ to be equal to $N^0$ and $\{I_4\}$ respectively. By Lemma \ref{Psiaction-starcondition} (2), $\psi=1$ on $N^0.$ Thus, \eqref{starcondition} holds for both $\sigma_1$ and $\sigma_2.$ We proceed to calculate $\Phi_w$ in both the cases.
	
	\noindent\textit{Case}(1): Fix $w = \sigma_1.$  We calculate the relevant subgroups as per \S \ref{geometriclemma} to compute the contribution. They are as follows:
	
	\begin{enumerate}
		\item $
		{M_{\psi}}'(w) =  wM{w}^{-1} \cap M_{\psi}= M^2.$

		\item $N'(w)= wN{w}^{-1} \cap M_{\psi} = M^1.$

		\item $N'_{w}  = {w}^{-1}Nw \cap M = M^1.$

		\item ${\psi}'_{w} : N'_{w} \rightarrow \C^{\times}$ is the character given by $X \mapsto \psi(wX{w}^{-1}),$ for every $X \in N'_{w}.$ Writing $X \in M^1$ as $X = \diag(n_y,n^{-y}),$ we get $wX{w}^{-1} = \left[\begin{array}{cc}
		I_2 & \left[\begin{array}{cc}
		0 & y \\
		y & 0
		\end{array}\right] \\
		&I_2  
		\end{array} \right],$ and 
		${\psi}'_{w}(X) = \psi_{_0}\left(\tr\left(C \left[\begin{array}{cc}
		0 & y \\
		y & 0
		\end{array}\right] \right)\right) =1.$ Hence, ${\psi}'_{w} = \one.$
		
		\item $
		M'(w)  = {w}^{-1}M_{\psi} w \cap M= M^2.$

		\item $  {w}^{-1}S_{\psi} w \cap N =N^1.$

		\item $  wP {w}^{-1} \cap N = N^1.$
	\end{enumerate}
	Next, we compute the character $\varepsilon(w)$ (see \S \ref{modularcharacters} (1)) of the group $M^2.$  Let an element $m \in M^2$ be written as $m={\rm diag}(a,d,a^{-1},d^{-1})$ where $a^2=1$ and $d\in F^{\times}.$ By Remark \ref{mod-char-MPsi} and Proposition \ref{modular-character-computation} (2), we have $[\varepsilon_1(w)](m)={\rm mod}_{N}^{\frac{1}{2}}(m) \cdot {\rm mod}_{N^1}^{-\frac{1}{2}}(m) =|d|_{_F}^{\frac{3}{2}} \cdot |d|_{_F}^{-\frac{3}{2}}=1.$ Similarly, $[\varepsilon_2(w)](m)=1.$ Hence $[\varepsilon(w)](m)=1$ for all $m\in M^2.$
	
	We calculate the action given by \eqref{inducingaction}. Write $m \in M^2$ as $m=\diag(a,d,a^{-1},d^{-1}).$ 
	Then, \begin{align*}
	[\varepsilon(w) \cdot r_{N'_{w},{\psi}'_{w}}(\rho)]^{w}(m) &=  \varepsilon(w)[w^{-1}mw] \cdot r_{N'_{w},{\psi}'_{w}}(\rho)[w^{-1}mw]  \\
	&=  \varepsilon(w)[\diag(a^{-1},d,a,d^{-1})]\cdot r_{M^1,\one}(\rho)[\diag(a^{-1},d,a,d^{-1})]  \\ 
	&= 1 \cdot r_{M^1,\one}(\rho)\left(\diag(a^{-1},d,a,d^{-1}) \right) \\
	&=r_{M^1,\one}(\rho)(m). 
\end{align*}
By \eqref{contributionPhi(w)}, $\Phi_{w}(\rho)=i_{_{M^1, \one}}(r_{M^1,\one}(\rho)).$  Since $M_\psi=M^2M^1,$ we can write $\Phi_{w}(\rho)={\rm mod}_{M^1}^{\frac{1}{2}}\cdot r_{M^1,\one}(\rho)$ by Remark \ref{NormalizationFactor} (1) (applied with $\mathcal{U}=M^{1}$ and $\mathcal{M}=M^2$), a representation of $M_{\psi}$ which acts trivially on $M^1.$ 
By \S \ref{relevantembeddings} (2), the action $r_{M^1,\one}(\rho)(m)$ can be identified with $r_{\overline{N_{1,1}}}(\rho)\left(\left[\begin{array}{cc}
	a &  0 \\
    0	& d   
	\end{array} \right] \right).$  Also, ${\rm mod}_{M^1}(m)=|d|_{_F}={\delta_{_{O_2(F,C)}}}\left(\left[\begin{array}{cc}
	a &  0 \\
	0	& d   
\end{array} \right] \right)$ by Proposition \ref{modular-character-computation} (5) and Remark \ref{mod-char-MPsi}. We conclude via \S \ref{relevantembeddings} (2) and (3) that  $\Phi_{w}(\rho)=\delta_{_{O_2(F,C)}}^{\frac{1}{2}} \cdot r_{\overline{N_{1,1}}}(\rho)_{|_{T_{2}(F,C)}},$ a representation of $O_2(F,C)$ with $\overline{N_{1,1}}$ action being trivial. Its action on elements of $O_2(F,C)$  is given by   
	$\Phi_{w}(\rho)\left(\left[\begin{array}{cc}
	\pm 1 & 0 \\
	y & d
	\end{array}\right] \right)=\delta_{_{O_2(F,C)}}^{\frac{1}{2}}\left(\left[\begin{array}{cc}
	\pm 1 & 0 \\
	0	& d
	\end{array}\right] \right) \cdot r_{\overline{N_{1, 1}}}(\rho)_{|_{T_{2}(F,C)}}\left(\left[\begin{array}{cc}
	\pm 1 & 0 \\
     0	& d
	\end{array}\right] \right).$ 
	Thus, $\Phi_{w}(\rho)$ is equal to the representation $(\delta_{_{O_{2}(F,C)}}^{\frac{1}{2}}  \cdot r_{\overline{N_{1, 1}}}(\rho))_{|_{T_{2}(F,C)}}$ of $T_2(F,C)$ regarded as a representation of $O_2(F,C)$ via Remark \ref{quotientO_2}.
	
\noindent \textit{Case}(2): Fix $w = \sigma_2.$ We have the following subgroups which feature in determining the contribution.
\begin{enumerate}
		\item $
		{M_{\psi}}'(w)  = wM{w}^{-1} \cap M_{\psi}
		=  M_{\psi}.$				
		\item $N'(w)
		= wN{w}^{-1} \cap M_{\psi} = \{I_4\}.$  		
		\item $N'_{w}  = {w}^{-1}Nw \cap M = \{I_4\}.$	
			
		\item 		${\psi}'_{w} = \one.$		
		
		\item $	M'(w) = {w}^{-1}M_{\psi} w \cap M
		= M_{\psi}.$		
		\item $  {w}^{-1}S_{\psi} w \cap N = \{I_4\}.$		
		\item $  wP {w}^{-1} \cap N = \{I_4\}.$
	\end{enumerate}
	We compute the character $\varepsilon(w)$ of $M_{\psi}$ (see \S \ref{modularcharacters} (1)). 
	 Suppose $X = \left[\begin{array}{cc}
	g &   \\
	& ^tg^{-1}   
	\end{array} \right] \in M_{\psi}$ with $g =\left[\begin{array}{cc}
	a & 0\\
	c & d
	\end{array}\right].$ By Remark \ref{mod-char-MPsi}, we have  $[\varepsilon_1(w)](X)=|d|_{_F}^{\frac{3}{2}}.$ Similarly $[\varepsilon_2(w)](X)=|d|_{_F}^{\frac{3}{2}}.$ Since $wXw^{-1}=\left[\begin{array}{cc}
	^tg^{-1} &   \\
	& g   
	\end{array} \right], [\varepsilon(w)](X) = [\varepsilon_1(w)](X) \cdot [\varepsilon_2(w)](wXw^{-1})=|d|_{_F}^{\frac{3}{2}} \cdot |d|_{_F}^{-\frac{3}{2}}=1.$

	We shall compute the action given by \eqref{inducingaction}. Write $X \in M_{\psi}$ as $X = \left[\begin{array}{cc}
	g &  \\
	& ^tg^{-1} \end{array} \right].$ Then, 
	\begin{align*}
	% \label{sigma2Action}
	[\varepsilon(w) \cdot r_{N'_{w},{\psi}'_{w}}(\rho)]^{w}(X) &=  \varepsilon(w)[w^{-1}Xw] \cdot r_{N'_{w},{\psi}'_{w}}(\rho)[w^{-1}Xw]  \\
	&=  \varepsilon(w)\left(\left[\begin{array}{cc}
	^tg^{-1} &  \\
	& g \end{array} \right]\right)\cdot r_{\{I_4\},\one}(\rho)\left(\left[\begin{array}{cc}
	^tg^{-1} &  \\
	& g \end{array} \right]\right) \\
	&= 1 \cdot r_{\{I_4\},\one}(\rho)\left(\left[\begin{array}{cc}
	^tg^{-1} &  \\
	& g \end{array} \right] \right) \\
    &= \rho(^tg^{-1}) [ \mbox{ by Remark \ref{NormalizationFactor} (2) } ].	\end{align*}
Let $(\rho_{_1},V)$ be  the representation of $GL_2(F)$ defined  by $\rho_{_1}(g)=\rho(^tg^{-1})$ for $g\in GL_2(F).$  Therefore, $[\varepsilon(w) \cdot r_{N'_{w},{\psi}'_{w}}(\rho)]^{w}(X)=\rho_{_1}(X)$ for all $X\in M_{\psi}.$
By \eqref{contributionPhi(w)}, $\Phi_{w}(\rho) = i_{_{\{I_4\}, \one}}(\rho_{_1}),$  which by Remark \ref{NormalizationFactor} (1) is further equal to $\rho_{_1}.$ It is well-known by a theorem of Gelfand-Kazhdan (\cite[Theorem 4.2.2]{Bump})  that $\rho_{_1}$ is equivalent to $\widetilde{\rho}$ as representations of $GL_2(F).$ By this theorem, $\Phi_{w}(\rho)\cong\widetilde{\rho}_{|_{M_{\psi}}}.$ Since $O_{2}(F,C)$ corresponds to $M_{\psi}$ in $Sp_{4}(F)$ under the map $\beta$ given by \eqref{beta}, $\Phi_{w}(\rho)= \widetilde{\rho}_{|_{O_{2}(F,C)}}.$
Thus, $r_{N,\psi}(\pi)$ is glued from the representations $\Phi_{\sigma_1}(\rho)= (\delta_{_{O_{2}(F,C)}}^{\frac{1}{2}}  \cdot r_{\overline{N_{1, 1}}}(\rho))_{|_{T_{2}(F,C)}} $ and  $\Phi_{\sigma_2}(\rho)= \widetilde{\rho}_{|_{O_{2}(F,C)}}$  completing \textbf{Step} (1).
\smallskip
	
\noindent \textbf{Step} (2). If $\rho$ is cuspidal, we of course have $\Phi_{\sigma_1}(\rho)=0.$ Hence, $r_{N,\psi}(\pi)=  \widetilde{\rho}_{|_{O_{2}(F,C)}}.$ On the other hand if $\rho$ is not cuspidal, $r_{\overline{N_{1, 1}}}(\rho)$ is non-zero and therefore the finer structure of $r_{N,\psi}(\pi)$ needs to be determined. Therefore, assume that $\rho$ is not cuspidal. Recall from Proposition \ref{S-psi-orbit-Sp4/P} that $ P \backslash Sp_{4}(F)= Orb_{S_{\psi}}(V_0) \bigsqcup Orb_{S_{\psi}}(\sigma_1 \cdot V_0) \bigsqcup Orb_{S_{\psi}}(\tau_1 \sigma_1 \cdot V_0) \bigsqcup Orb_{S_{\psi}}(\sigma_2 \cdot V_0).$ Note that $Orb_{S_{\psi}}(\sigma_2 \cdot V_0)=Orb_{P}(\sigma_2 \cdot V_0).$ Since $Orb_{P}(\sigma_2 \cdot V_0)$ is open in $P \backslash Sp_{4}(F)$ (via Remark \ref{open-orbits} (1)), we conclude that $Orb_{S_{\psi}}(\sigma_2 \cdot V_0)$ is open in $P \backslash Sp_{4}(F).$ Therefore, $\Phi_{\sigma_2}(\rho)$ appears as a subrepresentation (by second assertion of \cite[\S 5, Theorem 5.2]{BZ2}) of $r_{N,\psi}(\pi)$ and we have a short exact sequence:
	$$0 \to \Phi_{\sigma_2}(\rho) \to r_{N,\psi}(\pi) \to \rho_{_0} \to 0$$ where $\rho_{_0}$ is glued from representations $\Phi_{I_4}(\rho),$ $\Phi_{\sigma_1}(\rho)$ and $\Phi_{\tau_1\sigma_1}(\rho).$ But as $\Phi_{I_4}(\rho)=0=\Phi_{\tau_1\sigma_1}(\rho),$ we get $\rho_{_0}=\Phi_{\sigma_1}(\rho),$ which proves \eqref{structure-Siegel}. The proof of Theorem \ref{SiegelcaseforDegA-I} is now complete. \qedhere

\end{proof}

\subsection{Structure in the Klingen parabolic case}

In this section, we state and prove the second main result of our article which describes the structure of the twisted Jacquet module of a principal series representation induced from the Klingen parabolic. We make the following observations before we go to the main results.

\subsubsection{}\label{firstobservation} 
Consider the  following diagram. 

\begin{center}
	\usetikzlibrary{positioning,arrows.meta}
	\tikzcdset{row sep/normal=50pt, column sep/normal=50pt}
	
	\begin{tikzpicture}[
		on grid,
		every node/.style={anchor=base,minimum size=2mm},
		node distance=5mm and 3cm,
		]
		% First row
		\node (ab)                {\(N_{1,1}\)};
		\node (Rn)  [right=of ab] {\(1 \times N_{1,1}\)};
		\node (R)   [right=of Rn] {\(N^0\)};
		\node (t)   [right=of R] {\(N^2\)};
		\node (ft)  [right=of t]  {\(\C^{\times}\)};
		% second row
		\node (ftz) [below=of R] {\(X\)};
		\node (X)   [right=of ftz] {\(w X w^{-1}\)};
		\node (W)   [right=of X] {\(\psi(w X w^{-1})\)};
		% Arrows
		\draw [<->] (ab) to node [above,font=\scriptsize] {\(\)} (Rn);
		\draw [<->] (Rn) to node [above,font=\scriptsize] {\(\alpha\)} (R);
		\draw [->] (R) to node [above,font=\scriptsize] {\(w\)} (t);
		\draw [->] (t) to node [above] {\(\psi_{|_{N^2}}\)} (ft);
		\draw [->, bend left] (R) to node [above] {\(\psi_{\gamma}:=\psi_{|_{N^2}} \circ w\)} (ft);
		\draw [|->] (ftz) to (X);
		\draw [|->] (X) to (W); 
	\end{tikzpicture}
\end{center}
The map $\alpha$ is the embedding in \eqref{alpha} and by $w,$ we mean conjugation by the element $w=\tau_1$ or $w=\tau_1\sigma_1.$
We note that any character of the group  $N^0$ can be identified with a character of $N_{1,1}$ via $\alpha$ and hence with characters of $(F,+).$  Also, note that the group $N^2$ is the conjugate of $N^0$ under both $\tau_1$ and $\tau_1\sigma_1.$ Thus, the restriction of the character $\psi$ of $N$ given by \eqref{CharacterCorrspondngtoC} to $N^2$ can be thought of as a character of $N^0$ after conjugation by $\tau_1$ or $\tau_1\sigma_1$.   
Then,
\begin{equation*}
	(\psi_{|_{N^2}}\circ w) \left(\left[\begin{array}{cc}
		I_2 & \left[\begin{array}{cc}
			0 & 0 \\
			0 & x
		\end{array}\right]  \\
		& I_2 
	\end{array} \right]\right) =\psi_{\gamma}\left(\left[\begin{array}{cc}
		1 & x \\
		0 & 1
	\end{array}\right]\right) = \psi_{_0}(\gamma x)
\end{equation*} for every $x \in F,$ where $\psi_{\gamma}$ is the character of $N_{1,1}$ given by \eqref{Definitionpsi_x}. By abuse of notation, we denote the character $\psi_{|_{N^2}} \circ w$ of $N^0$ again by $\psi_{\gamma}.$

\subsubsection{}\label{part2-of-firstobservation}

Note that the groups $M^2$ and $M^3$  are naturally identified with $\{\pm 1 \} \times F^{\times}$ and $F^{\times} \times \{\pm 1\}$ respectively via $\alpha.$
In view of \S \ref{firstobservation}, the functor $r_{N^0, \psi_{\gamma}}: \Alg(L)\to \Alg(M^3)$ can be identified with the functor $r_{1\times N_{1,1},1\times \psi_{\gamma}}: \Alg(F^{\times} \times SL_2(F))\to \Alg(F^{\times}\times \{\pm 1\}).$ Let $D$ denote the subgroup $\{\diag(x,y,x^{-1},y^{-1}): x,y\in F^{\times}\}$ of $Sp_4(F).$ Then, the functor $r_{N^0, \one}: \Alg(L)\to \Alg(D)$ can be identified with the functor $r_{1\times N_{1,1},1\times \one}: \Alg(F^{\times} \times SL_2(F))\to \Alg(F^{\times}\times M_{1,1})$ where $M_{1,1}=\{\diag(y, y^{-1}): y\in F^{\times}\}.$ By composing with the restriction functor from $\Alg(D)$ to $\Alg(M^2),$ we get $r_{N^0,\one}:\Alg(L)\to \Alg(M^2)$ which identifies with $r_{1\times N_{1,1},1\times \one}: \Alg(F^{\times} \times SL_2(F))\to \Alg(\{\pm 1\}\times M_{1,1}).$ We note the following lemma in which we calculate conjugate of certain representations. These conjugates are representations of $M^2$ which arise in the calculation done in Theorem \ref{KlingencaseforDegA-I}.

\begin{lemma}\label{secondobservation} 
	Let $\eta$ be a character of $F^{\times}$ and $\tau\in \Irr(SL_2(F)).$ Let $d\in F^{\times}$ and $a^2=1.$ For each $\diag(a,d,a^{-1},d^{-1})\in M^2,$ the following hold:
	
	\begin{enumerate}
		\item  $[r_{N^0, \one}(\eta\otimes \tau)]^{\sigma_1}(\diag(a,d,a^{-1},d^{-1}))= \eta_{|_{\{\pm 1\}}}(a)\otimes r_{N_{1,1}}(\tau)(\diag(d,d^{-1})).$ 
		
		\item  $[r_{N^0, \psi_{\gamma}}(\eta\otimes \tau)]^{\tau_1}(\diag(a,d,a^{-1},d^{-1}))=r_{N_{1,1},\psi_{\gamma}}(\tau)(\pm I_2) \otimes \eta(d).$ 
		
		\item  $[r_{N^0, \psi_{\gamma}}(\eta\otimes \tau)]^{\tau_1 \sigma_1}(\diag(a,d,a^{-1},d^{-1}))= r_{N_{1,1},\psi_{\gamma}}(\tau)(\pm I_2) \otimes \eta^{-1}(d).$ 
	\end{enumerate} 	
\end{lemma}

\begin{proof}
	
	Write $m\in M^2$ as $m=\diag(a,d,a^{-1}, d^{-1})$ with $a^2=1$ and $d\in F^{\times}.$	
	For proof of (1), note that, $\sigma_1^{-1}m\sigma_1=\diag(a^{-1},d,a,d^{-1})=\diag(a,d,a^{-1},d^{-1})=m,$  i.e.,  $M^2$  is fixed point-wise via conjugation by $\sigma_1^{-1}.$ Hence, $[r_{N^0, \one}(\eta\otimes \tau)]^{\sigma_1}=r_{N^0, \one}(\eta\otimes \tau).$  By \S \ref{part2-of-firstobservation}, $[r_{N^0, \one}(\eta\otimes \tau)]^{\sigma_1}=\eta\otimes r_{N_{1,1}}(\tau)$ proving (1).
	We have $\tau_1^{-1}m\tau_1=\diag(d,a,d^{-1},a^{-1})$ and $(\tau_1\sigma_1)^{-1}m(\tau_1\sigma_1)=\diag(d^{-1},a,d,a^{-1}).$ Thus, for $w\in \{ \tau_1, \tau_1\sigma_1\},$ we have $w^{-1}M^2w=M^3.$ By \S \ref{part2-of-firstobservation}, $r_{N^0, \psi_{\gamma}}(\eta\otimes \tau)$ is equal to $\eta\otimes r_{N_{1,1},\psi_{\gamma}}(\tau).$ 
	We now prove (2) and (3). To prove (2),  fix $w = \tau_1$ and note that 
	\begin{align*}
	[r_{N^0, \psi_{\gamma}}(\eta\otimes \tau)]^{w}(\diag(a,d,a^{-1},d^{-1})) &=r_{N^0, \psi_{\gamma}}(\eta\otimes \tau)(\diag(d,a,d^{-1},a^{-1})) \\
	&=\eta(d)\otimes r_{N_{1,1}, \psi_{\gamma}}(\tau)(\diag(a,a^{-1})) \\
	&=r_{N_{1,1}, \psi_{\gamma}}(\tau)(\diag(a,a^{-1})) \otimes \eta(d).
	\end{align*} 
	Assertion (2) of the lemma follows now since $a^2=1$.
	For proof of (3), fix  $w= \tau_1\sigma_1.$ Then, 
	\begin{align*}
	[r_{N^0, \psi_{\gamma}}(\eta\otimes \tau)]^{w}(\diag(a,d,a^{-1},d^{-1}))& =r_{N^0, \psi_{\gamma}}(\eta\otimes \tau)(\diag(d^{-1},a,d,a^{-1}))\\
	&=\eta(d^{-1})\otimes r_{N_{1,1}, \psi_{\gamma}}(\tau)(\diag(a,a^{-1})) \\
	&=r_{N_{1,1}, \psi_{\gamma}}(\tau)(\diag(a,a^{-1})) \otimes \eta(d^{-1}). \qedhere 
	\end{align*} \end{proof}

\noindent We now present the second main result of this article.

\begin{theorem}\label{KlingencaseforDegA-I} 
	Fix a non-trivial character $\psi_{_0}$ of the additive group $F.$ Let $Q = L U$ denote the Klingen parabolic subgroup of the symplectic group $Sp_{4}(F).$ Suppose $\varrho = \eta \otimes \tau \in \Irr(L)$ where $\eta$ is a character of $F^{\times}$ and  $\tau \in \Irr(SL_{2}(F)).$ Denote $\Pi = \ind_{Q}^{Sp_{4}(F)}\left(\varrho\right).$ Let $\psi$ be the character of $N$ defined by \eqref{CharacterCorrspondngtoC}. Let $\varrho_{_0}$ and $\varrho_{_1}$ denote the representations of $O_{2}(F,C)$ obtained respectively from the representations $\delta_{_{O_{2}(F,C)}}^{\frac{1}{2}}  \cdot (r_{N_{1,1}, \psi_{\gamma}}(\tau) \otimes \eta)$ and $\delta_{_{O_{2}(F,C)}}^{\frac{1}{2}} \cdot (\eta_{_{|_{\{\pm{1}\}}}}\otimes  r_{N_{1,1}}(\tau))$ of $T_2(F,C)$ via composing with the quotient map. Let $\varrho_{_2}$ denote the representation
	$r_{N_{1,1}, \psi_{\gamma}}(\tau) \otimes \eta^{-1} $ of $T_2(F,C).$
	Then, $r_{N,\psi}(\Pi)$ is glued from the representations $\varrho_{_0},$ $\varrho_{_1}$ and 
	$\ind_{T_{2}(F,C)}^{O_{2}(F,C)}(\varrho_{_2})$ of $O_{2}(F,C).$ 
	
\end{theorem}

\begin{proof} 
	In view of Lemma \ref{Decomposability-verification-Klingen}, Proposition \ref{S-psi-orbit-Sp4/Q} and \S \ref{geometriclemma}, we may apply the Geometric Lemma to compute the twisted Jacquet module of $\Pi.$ 
By Proposition \ref{S-psi-orbit-Sp4/Q}, $\mathcal{W}= \{I_4, \tau_1,\sigma_1, \tau_1\sigma_1\}.$ To begin, we need to single out those $w \in \mathcal{W}$ for which \eqref{starcondition} holds.   For $w = I_4,$ we have $wUw^{-1} \cap N = N^3.$ 
By Lemma \ref{Psiaction-starcondition} (1), $\psi$ is non-trivial on $N^3$ and hence \eqref{starcondition} does not hold for $I_4$ yielding $\Phi_{I_4}=0.$ We shall consider each of the remaining $w,$ i.e., $w \in \{ \tau_1,\sigma_1, \tau_1\sigma_1\}$ one-by-one.
	
	\noindent \textit{Case}(1): Fix $w = \tau_1,$ we have $wUw^{-1} \cap N = N^1.$ By Lemma \ref{Psiaction-starcondition} (2), $\psi$ acts trivially on $N^1$ and hence \eqref{starcondition} holds for $\tau_1.$
	The subgroups relevant to compute the functor $\Phi_w$ are as follows:
	\begin{enumerate}
		\item $
		{M_\psi}'(w)  = wL{w}^{-1} \cap M_\psi 
		=M^2.$

		\item $U'(w)
		= wU{w}^{-1} \cap M_\psi = M^1.$

		\item $N'_{w}  = {w}^{-1}Nw \cap L =N^0.$

		\item ${\psi}'_{w} : N'_{w} \rightarrow \C^{\times}$ is the character given by $X \mapsto \psi(wX{w}^{-1}),$ for every $X \in N'_{w}.$ Explicitly, write $X \in N^0$ as $X = \left[\begin{array}{cc}
		I_2 & \left[\begin{array}{cc}
		0 & 0 \\
		0 & x
		\end{array}\right]  \\
		& I_2 
		\end{array} \right].$ Then, $wX{w}^{-1} = \left[\begin{array}{cc}
		I_2 & \left[\begin{array}{cc}
		x & 0 \\
		0 & 0
		\end{array}\right] \\
		&I_2  
		\end{array} \right]\in N^2$ and  
		${\psi}'_{w}(X) =  \psi(wX{w}^{-1})
		= \psi_{_0}\left(\tr\left(C \left[\begin{array}{cc}
		x & 0 \\
		0 & 0
		\end{array}\right] \right)\right) 
		=\psi_{_0}(\gamma x).$ Thus, ${\psi}'_{w}= \psi_{\gamma}$ ($\psi_{\gamma}$ as in \S \ref{firstobservation}).

		\item $L'(w)  = {w}^{-1}M_\psi w \cap L
		= M^3.$

		\item $  {w}^{-1}S_\psi w \cap U =U.$

		\item $ wQ {w}^{-1} \cap N = N.$ 
	\end{enumerate}
We compute the character $\varepsilon(w)$ of $M^3$ (see \S \ref{modularcharacters} (2)). 
Here, $w^{-1}S_\psi w \cap U = U.$ Therefore $\varepsilon_1(w)=1$ on $M^3.$ Similarly, $w Q w^{-1} \cap N = N$ which gives $\varepsilon_2(w)=1$ on $M^2.$ This yields  $\varepsilon(w)=1$ on $M^3.$	
	
We will determine the action given by \eqref{inducingaction}. Write $m \in M^2$ as $m=\diag(a,d,a^{-1},d^{-1})$ with $a^2=1,d\in F^{\times}.$ Then,  
	\begin{align*}
	[\varepsilon(w) \cdot r_{N'_{w},{\psi}'_{w}}(\varrho)]^{w}(m) &=  \varepsilon(w)[w^{-1}mw] \cdot r_{N'_{w},{\psi}'_{w}}(\varrho)[w^{-1}mw]  \\
	&= 1 \cdot [r_{N^0,\psi_{\gamma}}(\varrho)]^{w}(m). 
	\end{align*}
Thus, $[\varepsilon(w) \cdot r_{N'_{w},{\psi}'_{w}}(\varrho)]^{w}=[r_{N^0,\psi_{\gamma}}(\varrho)]^{w}.$  Recall from \S \ref{relevantembeddings} (3) that $M_{\psi}=M^2M^{1}.$  By \eqref{contributionPhi(w)} and Remark \ref{NormalizationFactor} (1) (applied with $\mathcal{M} = M^2, \mathcal{U} = M^1$ and $G = M_{\psi}$),  $\Phi_w(\varrho)= i_{_{M^1, \one}}([r_{N^0,\psi_{\gamma}}(\varrho)]^{w}) ={\rm mod}_{M^1}^{\frac{1}{2}} \cdot [r_{N^0,\psi_{\gamma}}(\varrho)]^{w}$, a representation of $M_{\psi}$ which acts trivially on $M^1.$  
But, $[r_{N^0,\psi_{\gamma}}(\varrho)]^{w}(m)=r_{N_{1,1}, \psi_{\gamma}}(\tau)(\diag(a,a^{-1})) \otimes \eta(d)$ by Lemma \ref{secondobservation} (2). By Proposition \ref{modular-character-computation} (5) and Remark \ref{mod-char-MPsi}, we have ${\rm mod}_{\overline{N_{1,1}}}\left(\left[\begin{array}{cc}
	a & 0 \\
	0 & d
\end{array}\right]\right)= |d|_{_F} ={\rm mod}_{M^1}(m).$  Consequently, by \S \ref{relevantembeddings} (2) and (3), $\Phi_{w}(\varrho)$ is equal to ${\rm mod}_{\overline{N_{1,1}}}^{\frac{1}{2}} \cdot (r_{N_{1,1},\psi_{\gamma}}(\tau) \otimes \eta),$ a representation of $O_2(F,C)$ with $\overline{N_{1,1}}$ action being trivial. Thus, $\Phi_{w}(\varrho)$ is equal to $	\delta_{_{O_{2}(F,C)}}^{\frac{1}{2}} \cdot (r_{N_{1,1},\psi_{\gamma}}(\tau) \otimes \eta),$ which is a representation of  $T_2(F,C)$ regarded as a representation of $O_2(F,C)$ via Remark \ref{quotientO_2}.	
	
For the remaining $w \in \mathcal{W}$ i.e., $w=\sigma_1$ and $w=\tau_1\sigma_1$ we have $wUw^{-1} \cap N = \left\{I_4\right\}.$ Thus, \eqref{starcondition} holds for both $\sigma_1$ and $\tau_1\sigma_1.$ We proceed to calculate $\Phi_w$ in these two remaining cases.

\textit{Case}(2): Fix $w = \sigma_1.$ The subgroups required to compute the contribution are as follows: 
		\begin{enumerate}
		\item $
		{M_\psi}'(w)  = wL{w}^{-1} \cap M_\psi 
		= M^2.$

		\item $U'(w)
		= wU{w}^{-1} \cap M_\psi = M^1.$

		\item $N'_{w}  = {w}^{-1}Nw \cap L =N^0.$

		\item 	Write $X \in N^0$ as $X = \left[\begin{array}{cc}I_2 & \left[\begin{array}{cc}	0 & 0 \\	0 & x
		\end{array}\right]  \\
		& I_2 
		\end{array} \right].$ Then, $wX{w}^{-1} = \left[\begin{array}{cc}
		I_2 & \left[\begin{array}{cc}
		0 & 0 \\
		0 & x
		\end{array}\right] \\
		&I_2  
		\end{array} \right].$ Therefore, 
		${\psi}'_{w}(X) =  \psi(wX{w}^{-1})
		=\psi_{_0}\left(\tr\left(C \left[\begin{array}{cc}
		0 & 0 \\
		0 & x
		\end{array}\right] \right)\right)=1.$ Thus, ${\psi}'_{w}= \one.$
		
		\item$
		L'(w) = {w}^{-1}M_\psi w \cap L 
		= M^2.$

		\item $  {w}^{-1}S_\psi w \cap U = N^4.$

		\item $ wQ {w}^{-1} \cap N = N^0.$ \end{enumerate}
	We compute the character $\varepsilon(w)$ of $M^2$ (see \S \ref{modularcharacters} (2)).
	Write $m$ in $M^2$ as $m={\rm diag}(a,d,a^{-1},d^{-1}).$  By Remark  \ref{mod-char-MPsi-Klingen-also} and Proposition \ref{modular-character-computation} (4), we obtain $[\varepsilon_1(w)](m)=|d|_{_F}^{-\frac{1}{2}}.$ By Remark \ref{mod-char-MPsi} and Proposition \ref{modular-character-computation} (1), we have $[\varepsilon_2(w)](m)=|d|_{_F}^{\frac{3}{2}} \cdot |d|_{_F}^{-1}=|d|_{_F}^{\frac{1}{2}}.$  Therefore $[\varepsilon(w)](m)=|d|_{_F}^{-\frac{1}{2}} \cdot |d|_{_F}^{\frac{1}{2}}=1$ for all $m\in M^2.$	
	
	We compute the action given by \eqref{inducingaction}. Let $m=\diag(a,d,a^{-1},d^{-1}) \in M^2.$ Then, 
	\begin{align*}
	[\varepsilon(w) \cdot r_{N'_{w},{\psi}'_{w}}(\varrho)]^{w}(m) &=  \varepsilon(w)[w^{-1}mw] \cdot r_{N'_{w},{\psi}'_{w}}(\varrho)[w^{-1}mw]  \\
	&= 1 \cdot [ r_{N^0,\one}(\varrho)]^{w}(m).	
	\end{align*}
By \eqref{contributionPhi(w)},  $\Phi_w(\varrho)= i_{_{M^1, \one}}([r_{N^0,\one}(\varrho)]^{w}).$ Proceeding as in \textit{Case}(1), we can conclude that $\Phi_w(\varrho)= {\rm mod}_{M^1}^{\frac{1}{2}} \cdot [r_{N^0,\one}(\varrho)]^{w},$ a representation of $M_{\psi}$ with $M^1$ action being trivial. Since, $[r_{N^0,\one}(\varrho)]^{w}(m)=(\eta_{|_{\{\pm 1\}}})(a)\otimes r_{N_{1,1}}(\tau)(\diag(d,d^{-1}))$ by Lemma \ref{secondobservation} (1), we get $\Phi_w(\varrho)= 	\delta_{_{O_{2}(F,C)}}^{\frac{1}{2}} \cdot (\eta_{|_{\{\pm 1\}}}\otimes r_{N_{1,1}}(\tau)),$  which is a representation of $T_2(F,C)$ regarded as a representation of $O_2(F,C)$ by composing with the quotient map.

	\textit{Case}(3): Fix $w = \tau_1 \sigma_1.$ We have the following subgroups relevant for determining the contribution.

	\begin{enumerate}
		\item $
		{M_\psi}'(w)  = wL{w}^{-1} \cap M_\psi 
		= M^2.$

		\item $U'(w)
		= wU{w}^{-1} \cap M_\psi = \left\{I_4 \right\}.$

		\item $N'_{w} = {w}^{-1}Nw \cap L = N^0.$

		\item  This is identical to (4) of \textit{Case}(1) and  thus ${\psi}'_{w}= \psi_{\gamma}$ ($\psi_{\gamma}$ as in \S \ref{firstobservation}).

		\item $
		L'(w)  = {w}^{-1}M_\psi w \cap L =  M^3.$

		\item $  {w}^{-1}S_\psi w \cap U = \{I_4\}.$

		\item $  wQ {w}^{-1} \cap N = N^2.$ \end{enumerate}
We compute the characters $\varepsilon_1(w),$ $\varepsilon_2(w)$ and $\varepsilon(w)$ specified in  \S \ref{modularcharacters} (2). 
We note that the character $\varepsilon(w)$ is not trivial in this case. Observe that $\varepsilon_1(w) = {\rm mod}_{U}^{\frac{1}{2}} \cdot {\rm mod}_{w^{-1}S_{\psi}w \cap U}^{-\frac{1}{2}}$ is a character of $M^3$ whereas $\varepsilon_2(w) = {\rm mod}_{N}^{\frac{1}{2}} \cdot {\rm mod}_{w Q w^{-1} \cap N}^{-\frac{1}{2}}$ is a character of the group $M^2.$ For $m={\rm diag}(a,d,a^{-1},d^{-1}) \in M^3,$ by Remark \ref{mod-char-MPsi-Klingen-also},  $[\varepsilon_1(w)](m)= (|a|_{_F}^{4})^{\frac{1}{2}} \cdot 1=|a|_{_F}^{2}.$ On the other hand,  for $m'={\rm diag}(a,d,a^{-1},d^{-1})\in M^2,$ Remark \ref{mod-char-MPsi} and Proposition \ref{modular-character-computation} (3) give  $[\varepsilon_2(w)](m')=|d|_{_F}^{\frac{3}{2}}.$ To calculate $\varepsilon(w),$ we consider $m=\diag(a,d,a^{-1},d^{-1})\in M^3$ and its conjugate $wmw^{-1}.$ We have $wmw^{-1}={\rm diag}(d,a^{-1},d^{-1},a).$ Note that $d^2=1$ and hence, $wmw^{-1} \in M^2.$  Thus,
	\begin{align*}
	[\varepsilon(w)](\diag(a,d,a^{-1},d^{-1})) &=[\varepsilon_1(w)](\diag(a,d,a^{-1},d^{-1}))\cdot [\varepsilon_2(w)](\diag(d,a^{-1},d^{-1},a))\\
	& =|a|_{_F}^{2} \cdot |a|_{_F}^{-\frac{3}{2}}=|a|_{_F}^{\frac{1}{2}}.
	\end{align*}

	We compute the action given by \eqref{inducingaction}. Write $m' \in M^2$ as $m'= {\rm diag}(a,d,a^{-1},d^{-1}).$ Then, 
	\begin{align*}
	[\varepsilon(w) \cdot r_{N'_{w},{\psi}'_{w}}(\varrho)]^{w}(m') &=  \varepsilon(w)[w^{-1}m'w] \cdot r_{N'_{w},{\psi}'_{w}}(\varrho)[w^{-1}m'w] \\
	&= \varepsilon(w)[\diag(d^{-1},a,d,a^{-1})] \cdot [r_{N^0,\psi_{\gamma}}(\varrho)]^{w}(m') \\
&=|d|_{_F}^{-\frac{1}{2}}\cdot [r_{N^0,\psi_{\gamma}}(\varrho)]^{w}(m')\\
&=(\delta_{M_{\psi}}^{-\frac{1}{2}}\cdot [r_{N^0,\psi_{\gamma}}(\varrho)]^{w})(m') [\mbox{ by Remark \ref{mod-char-MPsi} }]. \end{align*}
Thus, $[\varepsilon(w) \cdot r_{N'_{w},{\psi}'_{w}}(\varrho)]^{w}= (\delta_{M_{\psi}}^{-\frac{1}{2}}\cdot [r_{N^0,\psi_{\gamma}}(\varrho)]^{w}).$  By \eqref{contributionPhi(w)}, we have $\Phi_w(\varrho)= i_{_{\{I_4\}, \one}}(\delta_{M_{\psi}}^{-\frac{1}{2}}\cdot [r_{N^0,\psi_{\gamma}}(\varrho)]^{w})$ and it is further equal to $i_{M^2}^{M_{\psi}}(\delta_{M_{\psi}}^{-\frac{1}{2}}\cdot [r_{N^0,\psi_{\gamma}}(\varrho)]^{w})$ by Remark \ref{NormalizationFactor} (1). Translating to normalized induction, we obtain $\Phi_w(\varrho)=\ind_{M^2}^{M_{\psi}}([r_{N^0,\psi_{\gamma}}(\varrho)]^{w}).$ 
 By Lemma \ref{secondobservation} (3), $[r_{N^0,\psi_{\gamma}}(\varrho)]^{w}(m')= r_{N_{1,1}, \psi_{\gamma}}(\tau)(\diag(a,a^{-1})) \otimes \eta^{-1}(d).$ Thus, $\Phi_w(\varrho)$ is equal to the representation $\ind_{T_{2}(F,C)}^{O_{2}(F,C)}[r_{N_{1, 1},\psi_{\gamma}}(\tau) \otimes \eta^{-1}]$ by \S \ref{relevantembeddings} (2). From \textit{Case}(1), (2) and (3), $r_{N,\psi}(\Pi)$ is glued from $\Phi_{\tau_1}(\varrho)= 	\delta_{_{O_{2}(F,C)}}^{\frac{1}{2}}  \cdot (r_{N_{1, 1},\psi_{\gamma}}(\tau) \otimes \eta),$ $\Phi_{\sigma_1}(\varrho)= 	\delta_{_{O_{2}(F,C)}}^{\frac{1}{2}}  \cdot (\eta_{_{|_{\{\pm{1}\}}}} \otimes r_{N_{1, 1}}(\tau))$ and $\Phi_{\tau_1 \sigma_1}(\varrho).$ The proof of Theorem \ref{KlingencaseforDegA-I} is  complete.\qedhere
	
	\end{proof}

We have the following corollary to Theorem \ref{KlingencaseforDegA-I} which gives a finer structure of the twisted Jacquet module of $\ind_{Q}^{Sp_4(F)}(\eta\otimes \tau)$ depending on the nature of the inducing representation $\tau\in \Irr(SL_2(F)).$ 

\begin{corollary}\label{structure-Klingen}
	
	Fix a non-trivial character $\psi_{_0}$ of the additive group $F.$ Let $Q = L U$ denote the Klingen parabolic subgroup of the symplectic group $Sp_{4}(F).$ Suppose $\varrho = \eta \otimes \tau \in \Irr(L)$ where $\eta$ is a character of $F^{\times}$ and  $\tau \in \Irr(SL_{2}(F)).$ Denote $\Pi = \ind_{Q}^{Sp_{4}(F)}\left(\varrho\right).$ Let $\psi$ be the character of $N$ defined by \eqref{CharacterCorrspondngtoC}. Let $\omega_{\tau}$ denote the central character of $\tau.$ Define $\varrho_{_{0}}$ and $\varrho_{_{2}}$ as follows: $\varrho_{_{0}}$ is the representation of $O_{2}(F,C)$ obtained from the representation $	\delta_{_{O_{2}(F,C)}}^{\frac{1}{2}} \cdot (\omega_{\tau} \otimes \eta)$ of $T_{2}(F,C)$ via composing with the quotient map and $\varrho_{_{2}}$ is the representation $\omega_{\tau} \otimes \eta^{-1}$ of $T_2(F,C).$ Let $\varrho_{_1}$ be the representation of $O_2(F,C)$ obtained from the representation $	\delta_{_{O_{2}(F,C)}}^{\frac{1}{2}}  \cdot (\eta_{|_{\{\pm 1\}}} \otimes r_{N_{1,1}}(\tau))$ of $T_{2}(F,C)$ via composing with the quotient map. The following statements hold:
	
	\begin{enumerate}

		\item If $r_{N_{1,1},\psi_{\gamma}}(\tau)=0$ and $\tau$ is cuspidal,  $r_{N,\psi}(\Pi)=0.$
		
		\item If $r_{N_{1,1},\psi_{\gamma}}(\tau)=0$ and $\tau$ is not cuspidal,  $r_{N,\psi}(\Pi)= \varrho_{_1}.$ In particular, if $\tau$ is one dimensional, $r_{N,\psi}(\Pi)=  \varrho_{_1}.$
		
		\item Assume that $\tau$ is cuspidal and $r_{N_{1,1},\psi_{\gamma}}(\tau)\neq 0.$ Then, $r_{N,\psi}(\Pi)$ sits in the following short exact sequence of $O_{2}(F,C)$-modules:
		\begin{equation}\label{structure-in-Klingen-I}
		0\to  \ind_{T_{2}(F,C)}^{O_{2}(F,C)}( \varrho_{_{2}} ) \to r_{N,\psi}(\Pi)\to  \varrho_{_{0}} \to 0.
		\end{equation}
		
		\item Assume that $\tau$ is not cuspidal and $r_{N_{1,1},\psi_{\gamma}}(\tau)\neq 0.$ Then, $r_{N,\psi}(\Pi)$ is given by a filtration
		\begin{equation}\label{structure-in-Klingen-II}
		\{0\} \subset W_2\subset W_1\subset W_0=r_{N,\psi}(\Pi)
		\end{equation}
		where $W_0/W_1= \varrho_{_{0}},$ $W_1/W_2=  \varrho_{_1}$  and $W_2=\ind_{T_{2}(F,C)}^{O_{2}(F,C)}( \varrho_{_{2}} ).$
	\end{enumerate}

\end{corollary}

\begin{proof} 
By  Theorem \ref{KlingencaseforDegA-I}, $r_{N,\psi}(\Pi)$ is glued from the representations $\Phi_{\tau_1}(\varrho)=	\delta_{_{O_{2}(F,C)}}^{\frac{1}{2}}  \cdot (r_{N_{1, 1},\psi_{\gamma}}(\tau) \otimes \eta),$ $\Phi_{\sigma_1}(\varrho)=\delta_{_{O_{2}(F,C)}}^{\frac{1}{2}}  \cdot (\eta_{_{|_{\{\pm{1}\}}}} \otimes r_{N_{1, 1}}(\tau))$ and $\Phi_{\tau_1 \sigma_1}(\varrho)=\ind_{T_{2}(F,C)}^{O_{2}(F,C)}[r_{N_{1, 1},\psi_{\gamma}}(\tau) \otimes \eta^{-1}].$ Proofs of (1) and (2) are clear. We prove (3) and (4) together. The assumption that $r_{N_{1,1},\psi_{\gamma}}(\tau)\neq 0$ means $r_{N_{1,1},\psi_{\gamma}}(\tau)=\omega_{\tau}.$  By Proposition \ref{StabilizerfororbitsGeneral} (applied with $n=2$ and $k=1$), $\Pi_{|_P}$ is given by a filtration $0\subset \Pi_1\subset \Pi_0=\Pi_{|_P}$ where $\Pi_{0}/\Pi_{1}$ is isomorphic to $i_{D_0}^P( \delta_Q^{\frac{1}{2}} \cdot \varrho)$ and $\Pi_1$ is isomorphic to   
	$i_{D_1}^P([\delta_Q^{\frac{1}{2}} \cdot \varrho]^{\sigma_1}).$ By Lemma \ref{Spsi-OrbitinKlingenDegA}, it is clear that $r_{N,\psi}(\Pi_0/\Pi_1)$ is supported on the $(Q,S_{\psi})$-double cosets given by $I_4$ and $\tau_1.$ Hence, $r_{N,\psi}(\Pi_0/\Pi_1)$ is glued from the representations $\Phi_{I_4}(\varrho)$ and $\Phi_{\tau_1}(\varrho).$ Since $\Phi_{I_4}=0,$ $r_{N,\psi}(\Pi_0/\Pi_1)= \Phi_{\tau_1}(\varrho).$ Noting that $r_{N_{1,1},\psi_{\gamma}}(\tau)=\omega_{\tau}$ yields that $r_{N,\psi}(\Pi_0/\Pi_1)$ is equal to $\varrho_{_{0}}.$ By Remark \ref{open-orbits} (2), we get the following short exact sequence: 
	\begin{equation}\label{shotr-exactseq-I}
	0 \to r_{N,\psi}(\Pi_{1}) \to r_{N,\psi}(\Pi) \to  \varrho_{_{0}} \to 0.
	\end{equation}

By Lemma \ref{SpsiOrbitNontrivialKlingenGegA}, $D_1\backslash P=Orb_{P}(X_1)= Orb_{S_\psi}(X_1) \bigsqcup Orb_{S_\psi}(\tau_1 \cdot X_1).$ Since $Orb_{S_\psi}(\tau_1 \cdot X_1)$ is open in $D_1\backslash P,$ it is open in  $Q \backslash Sp_4(F)$ by Remark \ref{open-orbits} (2). Also, $Orb_{S_\psi}(\tau_1 \cdot X_1) = Orb_{S_\psi}(\tau_1\sigma_1 \cdot X_0)$  implies that 
$r_{N,\psi}(\Pi)$ contains $\Phi_{\tau_1 \sigma_1}(\varrho)$ as a subrepresentation by second part of \cite[Theorem 5.2]{BZ2}. It is also clear that $r_{N,\psi}(\Pi_1)$ is glued from $\Phi_{\tau_1 \sigma_1}(\varrho)$ and $\Phi_{\sigma_1}(\varrho).$  We obtain the following  short exact sequence: 
	\begin{equation}\label{shotr-exactseq-II}
	0 \to \Phi_{\tau_1 \sigma_1}(\varrho) \to r_{N,\psi}(\Pi_1) \to \Phi_{\sigma_1}(\varrho) \to 0.
	\end{equation}
	If $\tau$ is cuspidal, we have $r_{N_{1,1}}(\tau)=0.$ Hence, $\Phi_{\sigma_1}(\varrho)=0.$ By \eqref{shotr-exactseq-II}, $r_{N,\psi}(\Pi_1)=\Phi_{\tau_1 \sigma_1}(\varrho).$ Thus, \eqref{structure-in-Klingen-I} follows from \eqref{shotr-exactseq-I}. Proof of (3) is completed. On the other hand, if $\tau$ is not cuspidal, $r_{N_{1,1}}(\tau)$ is non-zero. Consequently, \eqref{structure-in-Klingen-II} follows from \eqref{shotr-exactseq-I} and \eqref{shotr-exactseq-II}. This proves (4) and completes the proof of the corollary.
\end{proof}

\section{The finite field case}\label{finite-field-case}

In this section, we deduce the analogue of our main theorems when the groups are defined over a finite field $\F_q$ where the characteristic is not equal to $2.$  We shall see that the statements of our main results, after suitable modification, holds true for the finite field case as well. Throughout this section, $\F_q$ shall denote a finite field with $q$ elements where $q$ is odd. As previously adopted, denote the character $\psi_{_C}$ defined by \eqref{CharacterCorrspondngtoC} by $\psi.$ We note the following remark.
\begin{remark}\label{finitefield-adapatation}
\begin{enumerate}
	\item By \cite[Appendix 3, pp.167-169]{Z}, the Geometric Lemma may be applied to the finite field case as well. The point to observe here is that the representation $r_{N,\psi}(\pi)$ is a direct sum of $\Phi_w$'s.
	
	\item We note that  Proposition \ref{O2C-acts-GmodBbar} holds without change if we replace a $p$-adic field with a finite field.
	
	\item 	All modular characters are trivial in the case of a finite field.

	\item The remaining results of Section \ref{prelims} and Section \ref{Orbits-Spsi-action} hold true for a finite field $\F_q.$  
	
	\item The Gelfand-Kazhdan theorem invoked in the proof of Theorem \ref{SiegelcaseforDegA-I} holds for the finite field case as well (\cite[Proposition 4.1.1]{Bump}).
	
\end{enumerate}
\end{remark}
We are ready to state the theorems in the finite field case. In view of Remark \ref{finitefield-adapatation}, the proofs of our results in the finite field case are essentially the same as in the case of the $p$-adic field.  We first state the result in the case of induction from the Siegel parabolic subgroup.

\begin{theorem}\label{SiegelcaseforDegA-finite}
	Fix a non-trivial character $\psi_{_0}$ of the additive group $\F_q.$ Let $P=MN$ 
	denote the Siegel parabolic subgroup of the symplectic group $Sp_{4}(\F_q)$ and 
	let $\rho\in \Irr(GL_{2}(\F_q))$ be regarded as a representation of the Levi 
	subgroup $M$ of $P.$ Denote $\pi=\Ind_P^{Sp_{4}(\F_q)}(\rho).$ Let $\psi$ 
	be the character of $N$ defined by \eqref{CharacterCorrspondngtoC}. The following statements hold:
	
	\begin{enumerate}
		\item If $\rho$ is not cuspidal,  $r_{N,\psi}(\pi)=		
		\widetilde{\rho}_{|_{O_{2}(\F_q,C)}} \bigoplus \rho_{_0},$ where $\rho_{_0}$ is the representation of $O_{2}(\F_q,C)$ obtained from $r_{\overline{N_{1, 1}}}(\rho)_{|_{T_{2}(\F_q,C)}}$ via composing with the quotient map.

		\item If $\rho$ is cuspidal, $r_{N,\psi}(\pi) = \widetilde{\rho}_{|_{O_{2}(\F_q,C)}}$ as $O_{2}(\F_q,C)$-modules. 
		
\end{enumerate} \end{theorem}

The analogue to Theorem \ref{KlingencaseforDegA-I} and Corollary \ref{structure-Klingen} can be made very explicit in the case of finite field. To this end, let us briefly recall the following facts from the theory of representations of $SL_2(\F_q)$ where $q$ is odd.  

\subsection{Representations of $SL_2(\F_q)$}\label{Representations-sl2}

We refer the reader to \cite[\S 5 \& \S 6]{Garrett} for this paragraph. Denote by $B$ the standard Borel subgroup of $SL_2(\F_q).$ One has the Levi decomposition  $B = M_{1,1} N_{1,1}$  where the Levi subgroup $M_{1,1}= \left\{\left[\begin{array}{cc}
a &0 \\
0& a^{-1}
\end{array}\right]: a \in {\F_q}^{\times}\right\}.$ For $m=\left[\begin{array}{cc}
a &0 \\
0& a^{-1}
\end{array}\right] \in M_{1,1}$ and $n=\left[\begin{array}{cc}
1 &x \\
0& 1
\end{array}\right] \in N_{1,1}$ we have $mnm^{-1}=\left[\begin{array}{cc}
1 &a^2 x \\
0& 1
\end{array}\right].$ Note that there are two inequivalent classes of characters of $N_{1,1}$ under the conjugation action by $M_{1,1}.$ Denote these two characters by $\Psi$ and $\Psi'.$
We shall identify the irreducible representations $\tau$ of $SL_2(\F_q)$ according to whether $r_{N_{1,1},\Psi}(\tau)$ and $r_{N_{1,1},\Psi'}(\tau)$ are zero or non-zero. 

\begin{enumerate}
	\item $SL_2(\F_q)$ has a unique one-dimensional representation namely the trivial character whose twisted Jacquet modules with respect to both $\Psi$ and $\Psi'$ are zero.
	
	\item There are $q-1$ irreducible representations $\tau$ for which both $r_{N_{1,1},\Psi}(\tau)$ and $r_{N_{1,1},\Psi'}(\tau)$ are non-zero. These are
	\begin{enumerate}
		\item $\frac{q-3}{2}$ irreducible principal series representations each having dimension $q+1,$
		
		\item one special representation $St_2$ having dimension $q,$ which is  the non-trivial component of $\Ind_B^{SL_2(\F_q)}(\one)$ and,
		
		\item  $\frac{q-1}{2}$ irreducible cuspidal representations each having dimension $q-1.$

	\end{enumerate}
	
	\item Let $\chi$ be an order $2$ character of ${\F_q}^{\times}.$ There are two inequivalent irreducible components of $\Ind_B^{SL_2(\F_q)}(\chi)$ such that each has dimension $\frac{q+1}{2}.$  Denote them by $\tau_1$ and $\tau_2$ such that they satisfy the following:
	\begin{enumerate}
		\item $r_{N_{1,1},\Psi}(\tau_1)$ is non-zero and $r_{N_{1,1},\Psi'}(\tau_1)=0$ and
		
		\item $r_{N_{1,1},\Psi}(\tau_2)=0$ and  $r_{N_{1,1},\Psi'}(\tau_2)$ is non-zero.
	\end{enumerate}

	\item There are two cuspidal representations of $SL_2(\F_q)$ each having dimension $\frac{q-1}{2}.$ We denote them by $\tau_1'$ and $\tau_2'$ such that 
	\begin{enumerate}
		\item $r_{N_{1,1},\Psi}(\tau_1')$ is non-zero and $r_{N_{1,1},\Psi'}(\tau_1')=0$ and
		
		\item $r_{N_{1,1},\Psi}(\tau_1')=0$ and $r_{N_{1,1},\Psi'}(\tau_2')$ is non-zero.
	\end{enumerate}
\end{enumerate}

\begin{remark} 
	We wish to note that if $r_{N_{1,1},\Psi}(\tau)$ or $r_{N_{1,1},\Psi'}(\tau)$ is non-zero then it is equal to $\omega_{\tau},$ the central character of $\tau.$
\end{remark}

We have the following theorem which is the analogue of Corollary \ref{structure-Klingen} in the case of induction from the Klingen parabolic.

\begin{theorem}\label{KlingencaseforDegA-finite}
	Fix a non-trivial character $\psi_{_0}$ of the additive group $\F_q.$ Let $Q = L U$ denote the Klingen parabolic subgroup of the symplectic group $Sp_{4}(\F_q).$ Suppose $\varrho = \eta \otimes \tau \in \Irr(L)$ where $\eta$ is a character of ${\F_q}^{\times}$ and $\tau \in \Irr(SL_{2}(\F_q)).$ Denote $\Pi = \Ind_{Q}^{Sp_{4}(\F_q)}\left(\varrho\right).$  Let $\psi$ be the character of $N$ defined by \eqref{CharacterCorrspondngtoC}. Let $\omega_{\tau}$ denote the central character of $\tau.$ Let $\varrho_{_{0}}$ and $\varrho_{_1}$  be the representations of $O_{2}(\F_q,C)$ obtained from the representations $\omega_{\tau} \otimes \eta$ and $\eta_{|_{\{\pm 1\}}} \otimes r_{N_{1,1}}(\tau)$ of $T_{2}(\F_q,C)$ via composing with the quotient map.  Let  $\varrho_{_{2}}$ denote the representation $\omega_{\tau} \otimes \eta^{-1}$ of $T_2(\F_q,C).$	Put $\mu_1= \varrho_{_{0}} \bigoplus \varrho_{_1} \bigoplus \Ind_{T_2(\F_q,C)}^{O_{2}(\F_q,C)}(\varrho_{_{2}})$ and $\mu_2= \varrho_{_{0}} \bigoplus \Ind_{T_2(\F_q,C)}^{O_{2}(\F_q,C)}(\varrho_{_{2}}).$ The following statements hold:
	
	\begin{enumerate}
		\item If $\tau$ is one dimensional, $r_{N,\psi}(\Pi)= \varrho_{_1}.$
		
		\item If $\tau$ is an irreducible principal series representation or the special representation $St_2,$  $r_{N,\psi}(\Pi) = \mu_1.$ 
		
		\item If $\tau$ is an irreducible cuspidal representation of dimension $q-1,$ $r_{N,\psi}(\Pi) = \mu_2.$ 
		
		\item Fix $\tau=\tau_1.$ If $\psi_{\gamma}$ is in the $M_{1,1}$-orbit of $\Psi$ then $r_{N,\psi}(\Pi)$ is equal to  $\mu_1$  and equal to $\varrho_{_1}$ otherwise.  
		
		\item Fix $\tau=\tau_1'.$ If $\psi_{\gamma}$ is in the $M_{1,1}$-orbit of $\Psi$  then $r_{N,\psi}(\Pi)$ is equal to  $\mu_2$  and equal to $0$ otherwise.		
	\end{enumerate}

\end{theorem}

\begin{remark}
	Statements analogous to (4) and (5) can be made for the representations $\tau_2$ and $\tau_2'$ as well.
	
\end{remark}

\section*{Acknowledgements}
This work forms a significant component  of Sanjeev Kumar Pandey's thesis at the Indian Institute of Science Education and Research Tirupati. The Ministry of Education, Government of India, is acknowledged by Sanjeev Kumar Pandey for providing financial
assistance through the Institute Fellowship program. The authors thank Dipendra Prasad for several helpful conversations and insightful comments during the preparation of this article. The authors would like to also thank U K Anandavardhanan and  Shiv Prakash Patel for their useful remarks.

\end{document}